\documentclass[11pt]{article}

\usepackage{amsmath,verbatim}
\usepackage{amsfonts}
\usepackage{amssymb}
\usepackage{amsthm,epic}
\usepackage{mathrsfs}
\usepackage{xcolor}
\usepackage{hyperref}
\usepackage{subcaption}

\usepackage{tikz} 
\usepackage[T1]{fontenc}
\usepackage{enumitem}
\usepackage{pgf, tikz, pstricks, pict2e}
\usepackage{graphicx, xcolor}
\usepackage{footnote,pifont,dsfont,mathrsfs}
\usepackage{multirow}
\usetikzlibrary{patterns}

\renewcommand{\geq}{\geqslant}
\renewcommand{\leq}{\leqslant}
\renewcommand{\epsilon}{\varepsilon}

\newtheorem{thm}{Theorem}

\newtheorem{proposition}[thm]{Proposition}
\newtheorem{lem}[thm]{Lemma}

\theoremstyle{definition}
\newtheorem{defn}[thm]{Definition}

\theoremstyle{remark}

\newtheorem{rem}[thm]{Remark}
\usepackage{array}
\usepackage[font={small}]{caption}

\theoremstyle{remark}

\usepackage[font={small,it}]{caption}

\colorlet{darkgreen}{green!50!black}
\definecolor{darkseagreen}{rgb}{0.56, 0.74, 0.56}
\definecolor{lightcyan}{rgb}{0.88, 1.0, 1.0}
\definecolor{lightblue}{rgb}{0.68, 0.85, 0.9}
\definecolor{palecerulean}{rgb}{0.61, 0.77, 0.89}
\definecolor{lgreen} {RGB}{180,210,100}
\definecolor{dblue}  {RGB}{20,66,129}
\definecolor{ddblue} {RGB}{11,36,69}
\definecolor{lred}   {RGB}{220,0,0}
\definecolor{nred}   {RGB}{224,0,0}
\definecolor{norange}{RGB}{230,120,20}
\definecolor{nyellow}{RGB}{255,221,0}
\definecolor{ngreen} {RGB}{98,158,31}
\definecolor{dgreen} {RGB}{78,138,21}
\definecolor{nblue}  {RGB}{28,130,185}
\definecolor{jblue}  {RGB}{20,50,100}
\definecolor{Apricot} {RGB}{255, 170, 123} 
\definecolor{dpurple}  {RGB}{53,21,93}

\DeclareMathOperator{\Ind}{ind}
\newcommand{\dz}{\,\mathrm{d}z}
\newcommand{\ds}{\,\mathrm{d}s}

\newcommand{\du}{\,\mathrm{d}u}

\makeatletter
\def\testb#1{\testb@i#1,,\@nil}%
\def\testb@i#1,#2,#3\@nil{%
  \draw[->, thick] (O) --++(#1);
  \ifx\relax#2\relax\else\testb@i#2,#3\@nil\fi}
\makeatother   
\newcommand{\makediag}[1]{
    \coordinate (O) at (0,0); \coordinate (N) at (0,0.8);
    \coordinate (NE) at (0.8,0.8); \coordinate (E) at (0.8,0);
    \coordinate (SE) at (0.8,-0.8); \coordinate (S) at (0,-0.8);
    \coordinate (SW) at (-0.8,-0.8);\coordinate (W) at (-0.8,0);
    \coordinate (N2E) at (1.6,0.8);\coordinate (S2W) at (-1.6,-0.8);
    \coordinate (NW) at (-0.8,0.8); \coordinate (B1) at (1.2,1.2);
    \coordinate (B2) at (-1.2,-1.2);
    \testb{#1}
} 
\newcommand{\diagr}[1]{
  \begin{tikzpicture}[scale=0.8]\makediag{#1}\end{tikzpicture}
}
\newcommand{\smalldiagr}[1]{
  \begin{tikzpicture}[scale=0.4]\makediag{#1}\end{tikzpicture}
}

\usepackage[T1]{fontenc}
\usepackage[mathscr]{eucal}
\usepackage{fullpage}

\makeatletter
\def\testbb#1{\testbb@i#1,,\@nil}%
\def\testbb@i#1,#2,#3\@nil{%
  \draw (O) --++(#1);
  \ifx\relax#2\relax\else\testbb@i#2,#3\@nil\fi}
\makeatother

\date{\today}
\author{K. Raschel\thanks{CNRS, Institut Denis Poisson, Universit\'e de Tours, France; \texttt{raschel@math.cnrs.fr}}\ \thanks{This project has received funding from the European Research Council (ERC) under the European Union's Horizon 2020 research and innovation programme under the Grant Agreement No 759702.} \& A. Trotignon\thanks{Department of Mathematics, Simon Fraser University, Canada \& Institut Denis Poisson, Universit\'e de Tours, France; \texttt{amelie.trotignon@idpoisson.fr}}}

\title{On walks avoiding a quadrant}

\begin{document}
\maketitle

\vspace{-4mm}

\begin{abstract}
Two-dimensional (random) walks in cones are very natural both in combinatorics and probability theory: they are interesting in their own right and also in relation to other discrete structures. While walks restricted to the first quadrant have been well studied, the case of planar, non-convex cones---equivalent to the three-quarter plane after a linear transform---has been approached only recently. In this article we develop an analytic approach for the enumeration of walks in three quadrants. The advantage of this method is the uniform treatment of models corresponding to different step sets. After splitting the three quadrants into two symmetric convex cones, the method is composed of three main steps: write a system of functional equations satisfied by the counting generating function, which may be simplified into one single equation under symmetry conditions; transform the functional equation into a boundary value problem; and finally solve this problem, using a new concept of anti-Tutte's invariant. The result is a contour-integral expression for the generating function. Such systems of functional equations also appear in queueing theory, namely, in the Join-the-Shortest-Queue model, which is still open in the non-symmetric case.
\end{abstract}

\begin{figure}[ht!]
\centering
\begin{tabular}{cc}
\includegraphics[scale=0.28]{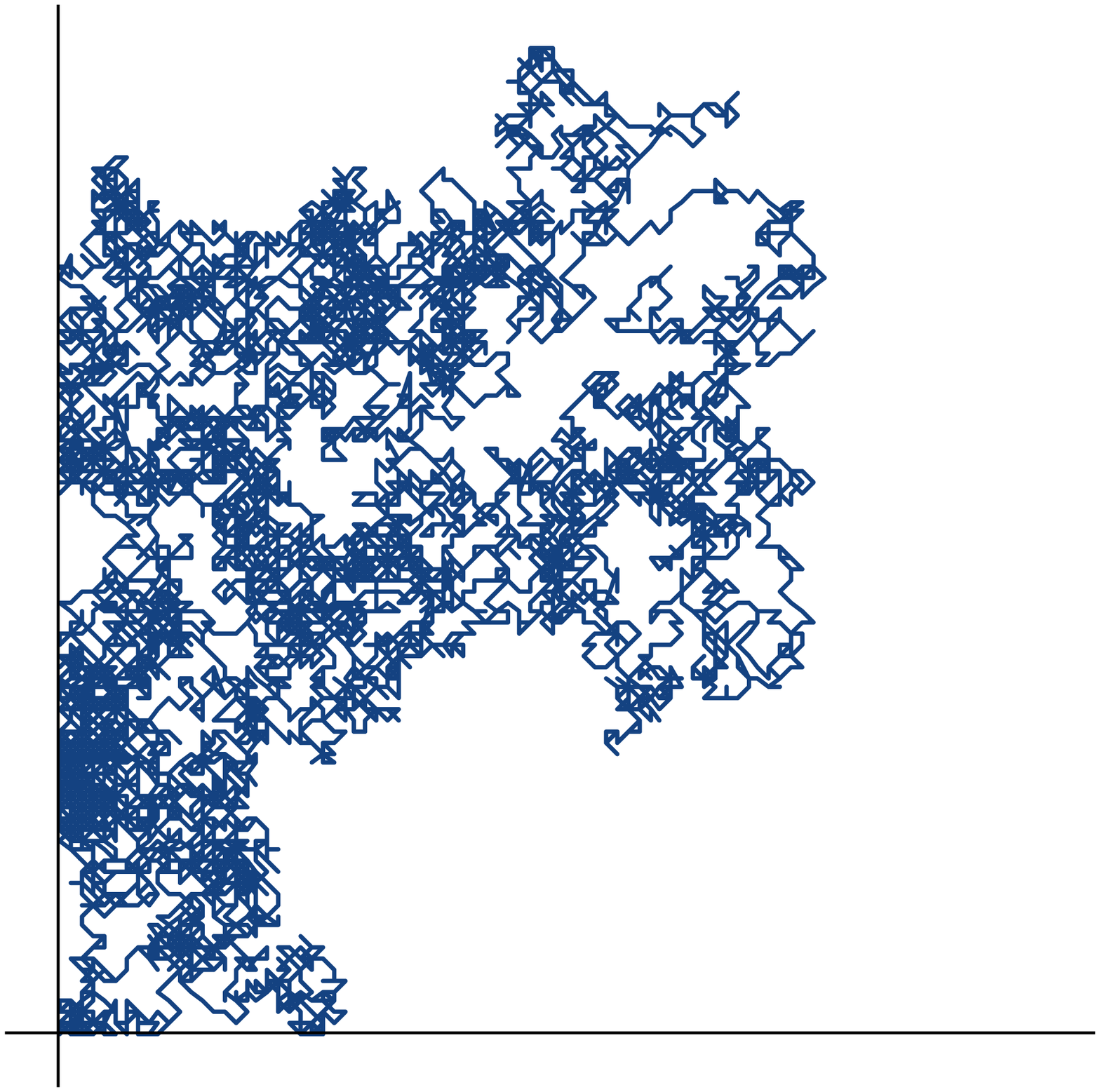}\qquad&\qquad
\includegraphics[scale=0.28]{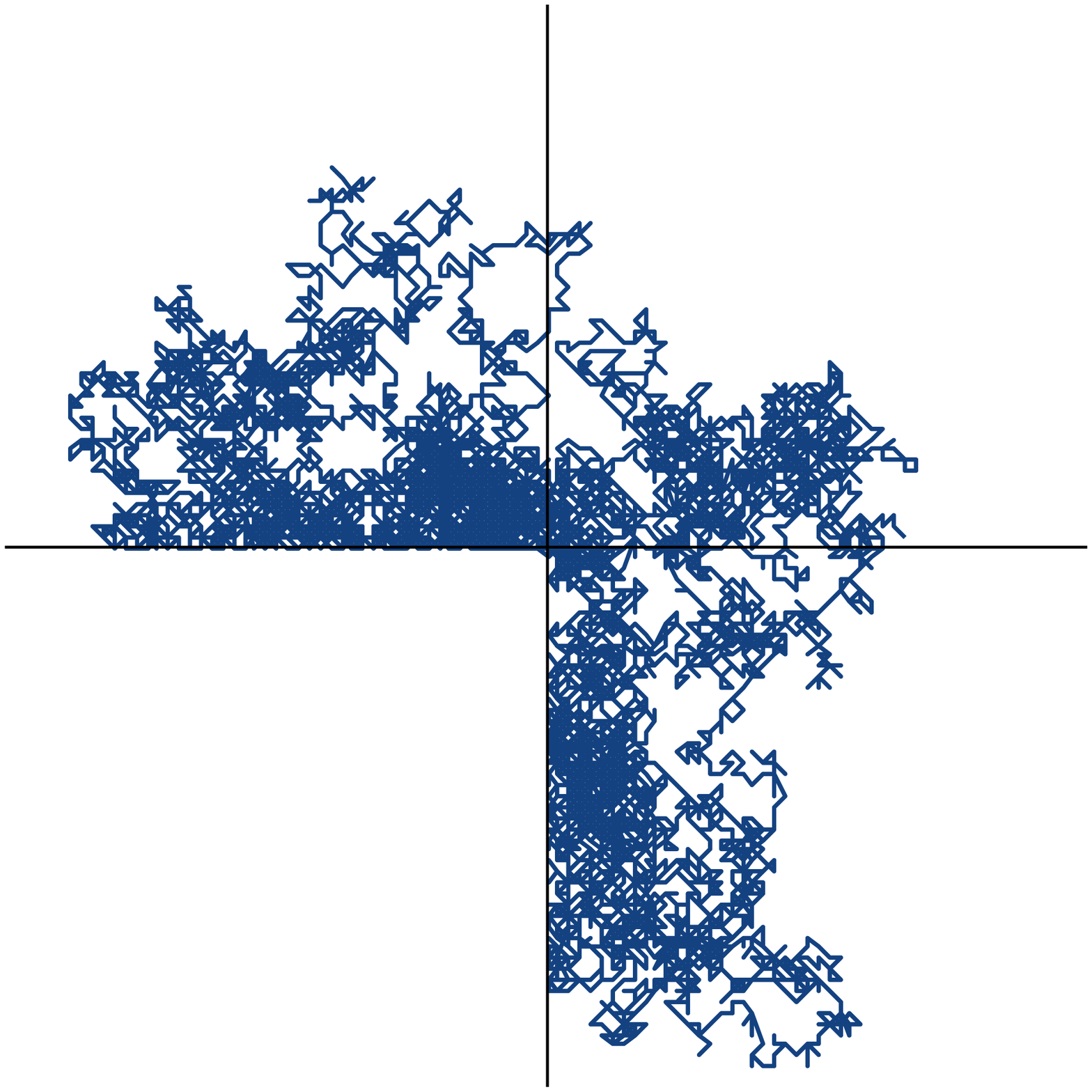}
\end{tabular}
\caption{Walks in the quarter plane and in three-quarter plane}
\label{fig:pictures_1/4_3/4}
\end{figure}

\newpage
\section{Introduction}

\paragraph{Context.}
Two-dimensional (random) walks in cones are very natural both in combinatorics and probability theory: they are interesting in their own right and also in relation to other discrete structures, see \cite{BMMi-10} and references therein. Most of the attention has been devoted to the case of convex cones (equivalent to the quarter plane, after a linear transform), see Figure~\ref{fig:pictures_1/4_3/4}, left. Thanks to an appealing variation of techniques, which complement and enrich each other (from combinatorics \cite{MiRe-09,BMMi-10}, complex analysis \cite{FaRa-12,Ra-12}, probability theory \cite{DeWa-15}, computer algebra \cite{BoKa-10,BoChHoKaPe-17}, Galois difference equations \cite{DrHaRoSi-17}), one now has a very good understanding of these quadrant models, most of the time via their generating function, which counts the number of walks of length $n$, starting from a fixed point, ending at an arbitrary point $(i,j)$ and remaining in the cone (see \eqref{eq:generating_function_1/4} below). Throughout the present work, all walk models will be assumed to have small steps, i.e., jumps in $\{-1,0,1\}^2$, see Figure \ref{fig:some_models} for a few examples. Let us recall a few remarkable results:
\begin{itemize}
     \item \textit{Exact expressions} exist for the generating function (to illustrate the variety of techniques, remark that the generating functions are infinite series in \cite{MiRe-09}, positive part extractions of diagonals in \cite{BMMi-10}, contour integrals on quartics in \cite{FaRa-12,Ra-12}, integrals of hypergeometric functions in \cite{BoChHoKaPe-17}, etc.);
     \item The \textit{algebraic nature of the trivariate generating function} is known: it is D-finite (that is, satisfies a linear differential equation with polynomial coefficients) if and only if a certain group of birational transformations is finite \cite{BMMi-10,BoKa-10,KuRa-12}. More recently, the differential algebraicity (existence of non-linear differential equations) of the generating function has also been studied \cite{BeBMRa-17,DrHaRoSi-17};
     \item The \textit{asymptotics of the number of excursions} (an excursion is a path joining two given points and remaining in the cone) \cite{BMMi-10,FaRa-12,DeWa-15,BoRaSa-14} is known. Although the full picture is still incomplete, the asymptotics of the total number of walks is also obtained in several cases \cite{BMMi-10,FaRa-12,DeWa-15,Du-14,BoChHoKaPe-17}.
\end{itemize}
Almost systematically, the starting point to solve the above questions is a functional equation that is satisfied by the generating function---it corresponds to the intuitive step-by-step construction of a walk, and will be stated later on, see \eqref{eq:functional_equation_3/4} and \eqref{eq:functional_equation_1/4}.

Given the vivid interest in combinatorics of walks confined to a quadrant, it is very natural to consider next the non-equivalent case of non-convex cones, as in particular the union of three quadrants 
\begin{equation*}
     \mathcal C=\{(i,j)\in\mathbb Z^2: i\geq0 \text{ or } j\geq0\},
\end{equation*}
see Figure \ref{fig:pictures_1/4_3/4}. Following Bousquet-M\'elou \cite{BM-16}, we will also speak about walks avoiding a quadrant. Although walks avoiding a quarter plane have many common features with walks in a quarter plane, the former cited model is definitely much more complicated. To illustrate this fact, let us recall \cite{BM-16} that the simple walk (usually the simplest model, see Figure \ref{fig:some_models}) in three quadrants has the same level of complexity as the notoriously difficult Gessel's model \cite{BoKa-10,BMGessel-16} in the quadrant!

\begin{figure}[htb]
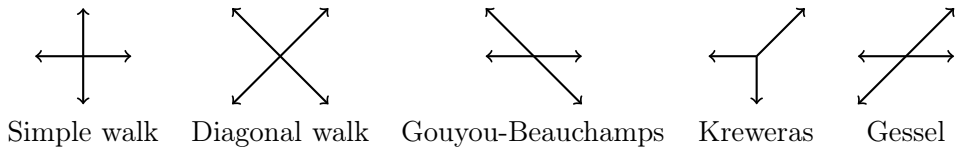

  \centering
  \begin{tabular}{ccccc}
    $\ \diagr{E,W,N,S}\ $ &$\ \diagr{NE,NW,SE,SW}\ $ &$\ \diagr{E,W,SE,NW}\ $ & $\ \diagr{NE,S,W}\ $  & $\ \diagr{NE,W,E,SW}\ $
\\
Simple walk&Diagonal walk&Gouyou-Beauchamps&Kreweras&Gessel
  \end{tabular}
   \caption{Some famous models of planar walks. The first two ones are solved in the three-quadrant in \cite{BM-16}}
  \label{fig:some_models}
\end{figure}

Three-quadrant walks have been approached only recently. In \cite{BM-16}, Bousquet-M\'elou solves the simple walk and the diagonal walk (see Figure \ref{fig:some_models} for a representation of these step sets) starting at various points. She obtains an exact expression of the generating function and derives several interesting combinatorial identities, among which a new proof of Gessel's conjecture via the reflection principle. Mustapha \cite{Mu-19} computes the asymptotics of the number of excursions for all small step models, following \cite{DeWa-15,BoRaSa-14} (interestingly and in contrast with combinatorics, the probabilistic results \cite{DeWa-15,Mu-19} on random walks in cones do not really depend on convexity). Using an original connection with planar maps, Budd \cite{Bu-17} obtains various enumerating formulas for planar walks, keeping track of the winding angle. These formulas can be used to enumerate simple walks in the three-quarter plane \cite{Bu-17,MiSi-18}. As recalled in \cite{BM-16}, the problem of diagonal walks on the square lattice was also raised in 2001 by David W.\ Wilson in entry A060898 of the OEIS \cite{OEIS}.

In this article we develop the analytic approach of \cite{FaIa-79,FaIaMa-17,Ra-12} to walks in three quadrants, thereby answering to a question of Bousquet-M\'elou in \cite[Sec.\ 7.2]{BM-16}.

\paragraph{Strategy.}
Once a step set $\mathcal S$ is fixed, our starting point is a functional equation satisfied by the generating function
\begin{equation}
\label{eq:generating_function_3/4}
     C(x,y)=\sum_{n\geq0}\sum_{(i,j)\in\mathcal C}c_{i,j}(n)x^{i}y^{j} t^{n},
\end{equation}
where $c_{i,j}(n)$ counts $n$-step $\mathcal S$-walks going from $(0,0)$ to $(i,j)$ and remaining in $\mathcal C$. Stated in \eqref{eq:functional_equation_3/4}, this functional equation translates the step-by-step construction of three-quadrant walks and takes into account the forbidden moves which would lead the walk into the forbidden negative quadrant. At first sight, this equation is very similar to its one-quadrant analogue (we will compare the equations \eqref{eq:functional_equation_3/4} to \eqref{eq:functional_equation_1/4} in Section \ref{subsec:kernel_functional_equations}), the only difference is that negative powers of $x$ and $y$ arise: this can be seen in the definition of the generating function \eqref{eq:generating_function_3/4} and on the functional equation \eqref{eq:functional_equation_3/4} as well, since the right-hand side of the latter involves some generating functions in the variables $\frac{1}{x}$ and $\frac{1}{y}$. This difference is fundamental and the methodology of \cite{BMMi-10,Ra-12} (namely, performing algebraic substitutions or evaluating the functional equation at well-chosen complex points) breaks down, as the series are no longer convergent.

\begin{figure}[htb]
\centering
\begin{tikzpicture}
\begin{scope}[scale=0.3]
\draw[white, fill=dblue!15] (0,-1) -- (5.5,4.5) -- (5.5,-5.5) -- (0,-5.5);
\draw[white, fill=dgreen!15] (-1,0) -- (-5.5,0) -- (-5.5,5.5) -- (4.5,5.5);
\draw[white, thick] (0,-5.5) grid (5.5,5.5);
\draw[white, thick] (-5.5,0) grid (0,5.5);
\draw[Apricot!90, thick] (0,0) -- (5.5,5.5);
\draw[->] (0,-5.5) -- (0,5.5);
\draw[->] (-5.5,0) -- (5.5,0);
\end{scope}
\end{tikzpicture}
   \caption{Splitting of the three-quadrant cone in two wedges of opening angle $\frac{3\pi}{4}$}
  \label{fig:three-quarter_split}
\end{figure}
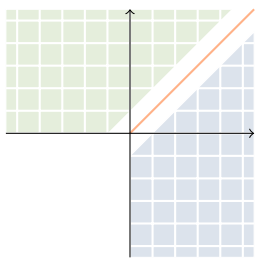

The idea in \cite{BM-16} is to see $\mathcal C$ as the union of three quarter planes, and to state for each quadrant a new equation, which is more complicated but (by construction) may be evaluated. Our strategy follows the same line: we split the three-quadrant in two convex cones (of opening angle $\frac{3\pi}{4}$, see Figure \ref{fig:three-quarter_split}) and write a system of two functional equations, one for each domain. The drawbacks of this decomposition is that it increases the complexity:
\begin{itemize}
     \item There are two functional equations instead of one;
     \item The functional equations involve more unknowns (corresponding to the diagonal and close-to-diagonal terms) in their right-hand sides, see Appendix \ref{app:proof_lem1}.
\end{itemize}
On the other hand:
\begin{itemize}
     \item The fundamental advantage is that the new equations may be evaluated---and ultimately will be solved;
     \item Unexpectedly, this splitting of the cone allows us to relate the combinatorial model of walks avoiding a quadrant to an interesting class of space inhomogeneous walks, among which a well-known problem in queueing theory: the Join-the-Shortest-Queue (JSQ) model, see Figure~\ref{fig:JSQ}.
     \end{itemize}
     
     \paragraph{Three-quadrant walks and space inhomogeneous walks.}
     Doing two simple changes of variables (one for each wedge, see in particular \eqref{eq:change_var}), the decomposition of the three-quarter plane shown on Figure \ref{fig:three-quarter_split} is equivalent to splitting a half-plane into the union of two quadrants and a half-line, see Figure \ref{fig:half-plane_split}. We end up with a \textit{space inhomogeneous model} in the half-plane. On the $y$-axis, the step set is composed of mixed steps from the step sets of the left and right quadrants. In particular, starting with a symmetric step set in the three-quarter plane, say the simple walk, one obtains (with the terminology of Figure \ref{fig:some_models}) Gouyou-Beauchamps' model in the left quadrant and Gessel's model in the right one, see Figure \ref{fig:half-plane_split_symm} on the left. This model is equivalent to study Gessel's step set in the quadrant, killed on the $x$-axis and reflected on the $y$-axis, see Figure \ref{fig:half-plane_split_symm} in the middle. A related model is studied in \cite{BeOwRe-18,XuBeOw-19}: in these articles, the authors work on walks in the quadrant with different weights on the boundary, see Figure \ref{fig:half-plane_split_symm} on the right, and give some results on the nature of the generating function of such walks.
     
     A related, simpler model (that we don't solve in the present paper) would be to split the full plane into two half-planes and a boundary axis, to consider in each of the three regions a (different) step set, and to solve the associated walk model, see Figure \ref{fig:plane_split}, right.
     
Some other space inhomogeneous walk models have been investigated in \cite{BoMuSi-15,Mu-19,BuKa-19}, but this notion of inhomogeneity does not match with ours. Indeed, a simple but typical example in \cite{BoMuSi-15,Mu-19,BuKa-19} consists in dividing $\mathbb Z^2$ into the odd and even lattices, and to assign to each point of the even (resp.\ odd) lattice a certain step set $\mathcal S$ fixed a priori (resp.\ another step set $\mathcal S'$), see Figure \ref{fig:plane_split}, left.

\begin{figure}[htb]
\centering
\begin{tikzpicture}
\begin{scope}[scale=0.45, xshift=0cm, yshift=2.5cm]
\draw[white, fill=dblue!15] (0,-1) -- (5.5,4.5) -- (5.5,-5.5) -- (0,-5.5);
\draw[white, fill=dgreen!15] (-1,0) -- (-5.5,0) -- (-5.5,5.5) -- (4.5,5.5);
\draw[white, thick] (0,-5.5) grid (5.5,5.5);
\draw[white, thick] (-5.5,0) grid (0,5.5);
\draw[Apricot!90, thick] (0,0) -- (5.5,5.5);
\draw[dashed, dgreen!90, thick] (-1,0) -- (4.5,5.5);
\draw[dashed, dblue!90, thick] (0,-1) -- (5.5,4.5);
\draw[->] (0,-5.5) -- (0,5.5);
\draw[->] (-5.5,0) -- (5.5,0);
\draw[thick, dblue, ->] (3,-3) -- (2,-3);
\draw[thick, dblue, ->] (3,-3) -- (2,-4);
\draw[thick, dblue, ->] (3,-3) -- (3,-4);
\draw[thick, dblue, ->] (3,-3) -- (4,-2);
\draw[thick, dblue, ->] (3,-3) -- (3,-2);
\draw[thick, dgreen, ->] (-3,3) -- (-4,3);
\draw[thick, dgreen, ->] (-3,3) -- (-3,4);
\draw[thick, dgreen, ->] (-3,3) -- (-2,4);
\draw[thick, dgreen, ->] (-3,3) -- (-3,2);
\draw[thick, dgreen, ->] (-3,3) -- (-4,2);
\draw[thick, nred, ->] (3,3) -- (2,3);
\draw[thick, nred, ->] (3,3) -- (2,2);
\draw[thick, nred, ->] (3,3) -- (3,2);
\draw[thick, nred, ->] (3,3) -- (4,4);
\draw[thick, nred, ->] (3,3) -- (3,4);
\end{scope}

\begin{scope}[scale=0.5, xshift=19.5cm, yshift=0cm]
\draw[white, fill=dblue!15] (1,0) -- (1,6.5) -- (6.5,6.5) -- (6.5,0);
\draw[white, fill=dgreen!15] (-1,0) -- (-6.5,0) -- (-6.5,6.5) -- (-1,6.5);
\draw[white, thick] (-6.5,0) grid (6.5,6.5);
\draw[dashed, dgreen!90, thick] (-1,0) -- (-1,6.5);
\draw[dashed, dblue!90, thick] (1,0) -- (1,6.5);
\draw[->] (0,-0.5) -- (0,6.5);
\draw[->] (-6.5,0) -- (6.5,0);
\draw[Apricot!90, thick] (0,0) -- (0,6.5);
\draw[thick, dblue, ->] (4,4) -- (3,4);
\draw[thick, dblue, ->] (4,4) -- (4,5);
\draw[thick, dblue, ->] (4,4) -- (5,5);
\draw[thick, dblue, ->] (4,4) -- (5,4);
\draw[thick, dblue, ->] (4,4) -- (4,3);
\draw[thick, nred, ->] (0,4) -- (1,4);
\draw[thick, nred, ->] (0,4) -- (-1,4);
\draw[thick, nred, ->] (0,4) -- (0,3);
\draw[thick, nred, ->] (0,4) -- (0,5);
\draw[thick, nred, ->] (0,4) -- (-1,5);
\draw[thick, dgreen, ->] (-4,4) -- (-5,4);
\draw[thick, dgreen, ->] (-4,4) -- (-4,5);
\draw[thick, dgreen, ->] (-4,4) -- (-4,3);
\draw[thick, dgreen, ->] (-4,4) -- (-5,5);
\draw[thick, dgreen, ->] (-4,4) -- (-3,3);
\end{scope}
\end{tikzpicture}
   \caption{Solving a walk model in the three-quadrant cone is equivalent to solving a space inhomogeneous model in a half-plane}
  \label{fig:half-plane_split}
\end{figure}
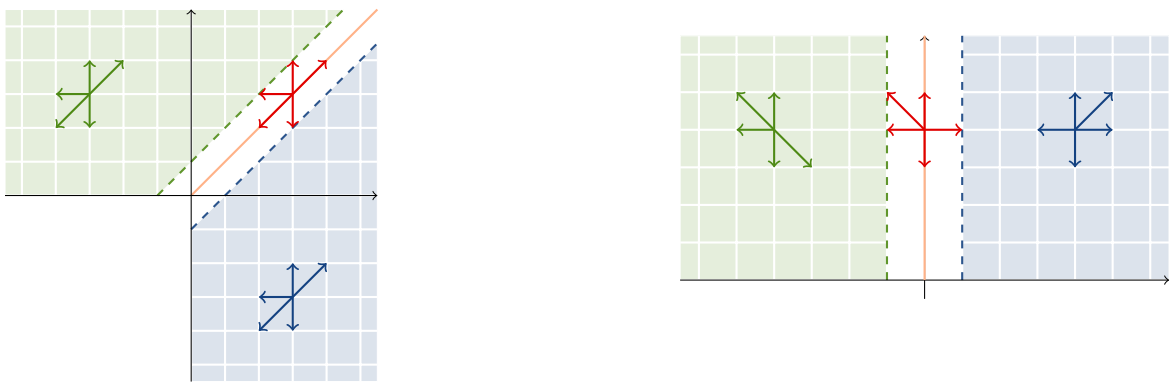

\begin{figure}[htb]
\centering
\begin{tikzpicture}
\begin{scope}[scale=0.5, xshift=0cm, yshift=0cm]
\draw[white, fill=dblue!15] (1,0) -- (1,6.5) -- (6.5,6.5) -- (6.5,0);
\draw[white, fill=dgreen!15] (-1,0) -- (-6.5,0) -- (-6.5,6.5) -- (-1,6.5);
\draw[white, thick] (-6.5,0) grid (6.5,6.5);
\draw[dashed, dgreen!90, thick] (-1,0) -- (-1,6.5);
\draw[dashed, dblue!90, thick] (1,0) -- (1,6.5);
\draw[->] (0,-0.5) -- (0,6.5);
\draw[->] (-6.5,0) -- (6.5,0);
\draw[Apricot!90, thick] (0,0) -- (0,6.5);
\draw[thick, dblue, ->] (4,4) -- (3,4);
\draw[thick, dblue, ->] (4,4) -- (5,4);
\draw[thick, dblue, ->] (4,4) -- (5,5);
\draw[thick, dblue, ->] (4,4) -- (3,3);
\draw[thick, nred, ->] (0,4) -- (1,4);
\draw[thick, nred, ->] (0,4) -- (-1,4);
\draw[thick, nred, ->] (0,4) -- (1,5);
\draw[thick, nred, ->] (0,4) -- (-1,5);
\draw[thick, dgreen, ->] (-4,4) -- (-3,4);
\draw[thick, dgreen, ->] (-4,4) -- (-5,4);
\draw[thick, dgreen, ->] (-4,4) -- (-3,3);
\draw[thick, dgreen, ->] (-4,4) -- (-5,5);
\end{scope}

\begin{scope}[scale=0.5, xshift=9cm, yshift=0cm]
\draw[white, fill=dpurple!15] (0,0) -- (0,6.5) -- (6.5,6.5) -- (6.5,0);
\draw[white, thick] (0,0) grid (6.5,6.5);
\draw[->] (0,0) -- (0,6.5);
\draw[->] (0,0) -- (6.5,0);
\draw[dpurple,thick,->] (4,4) -- (5,4);
\draw[dpurple,thick,->] (4,4) -- (5,5);
\draw[dpurple,thick,->] (4,4) -- (3,4);
\draw[dpurple,thick,->] (4,4) -- (3,3);
\draw[dpurple,thick,->] (4,0) -- (5,0);
\draw[dpurple,thick,->] (4,0) -- (5,1);
\draw[dpurple,thick,->] (4,0) -- (3,0);
\draw[dpurple,thick,->>] (0,4) -- (1,4);
\draw[dpurple,thick,->>] (0,4) -- (1,5);
\draw[dpurple,thick,->>] (0,0) -- (1,0);
\draw[dpurple,thick,->>] (0,0) -- (1,1);
\end{scope}

\begin{scope}[scale=0.5, xshift=19.5cm, yshift=0cm]
\draw[white, fill=gray!25] (0,0) -- (0,6.5) -- (6.5,6.5) -- (6.5,0);
\draw[white, thick] (0,0) grid (6.5,6.5);
\draw[->] (0,0) -- (0,6.5);
\draw[->] (0,0) -- (6.5,0);
\draw[thick, black] (0,0)--(1,1);
\draw[thick, black] (1,1)--(2,1);
\draw[thick, black] (2,1)--(1,0);
\draw[thick, black] (1,0)--(4,0);
\draw[thick, black] (4,0)--(5,1);
\draw[thick, black] (5,1)--(3,1);
\draw[thick, black] (3,1)--(5,3);
\draw[thick, black] (5,3)--(2,3);
\draw[thick, black] (2,3)--(1,2);
\draw[thick, black] (1,2)--(0,2);
\draw[thick, black] (0,2)--(2,4);
\draw[thick, black] (2,4)--(3,4);
\draw[thick, black] (3,4)--(4,5);
\draw[thick, black] (4,5)--(3,5);
\draw[thick, black] (3,5)--(4,6);
\draw[thick, black,->] (4,6)--(6,6);
\foreach \x in {1,2,...,6}  \draw[thick, dblue, fill=dblue]  (\x,0) circle (0.1cm);
\foreach \y in {1,2,...,6} \draw[thick, dgreen, fill=dgreen]  (0,\y) circle (0.1cm);
\draw[thick, nred, fill=nred]  (0,0) circle (0.1cm);
\end{scope}
\end{tikzpicture}
   \caption{Solving a symmetric walk model in the three-quadrant cone is equivalent to solving a walks partly killed and reflected in the quarter plane (pictures on the left and in the middle). Walks in the quarter plane with weights on the boundary (right)}
  \label{fig:half-plane_split_symm}
\end{figure}
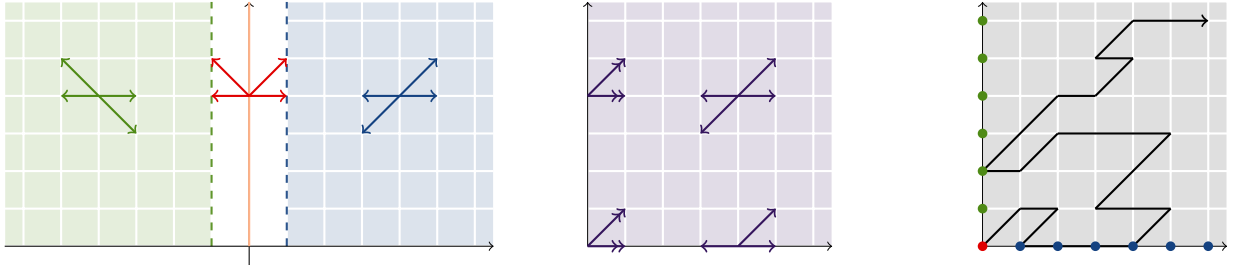

\begin{figure}[htb]
\centering
\begin{tikzpicture}
\begin{scope}[scale=0.45]
\draw[white, fill=gray!20] (-5.5,-5.5) -- (-5.5,5.5) -- (5.5,5.5) -- (5.5,-5.5);
\draw[white, thick] (-5.5,-5.5) grid (5.5,5.5);
\draw[->] (0,-5.5) -- (0,5.5);
\draw[->] (-5.5,0) -- (5.5,0);
\foreach \y in {-4,-2,0,2,4} \foreach \x in {-5,-3,-1,1,3,5} \draw[thick, nred, fill=nred]  (\x,\y) circle (0.1cm);
\foreach \y in {-5,-3,-1,1,3,5} \foreach \x in {-4,-2,0,2,4} \draw[thick, nred, fill=nred]  (\x,\y) circle (0.1cm);
\foreach \y in {-4,-2,0,2,4} \foreach \x in {-4,-2,0,2,4} \draw[thick, dblue, fill=dblue]  (\x,\y) circle (0.1cm);
\foreach \y in {-5,-3,-1,1,3,5} \foreach \x in {-5,-3,-1,1,3,5} \draw[thick, dblue, fill=dblue]  (\x,\y) circle (0.1cm);
\end{scope}

\begin{scope}[scale=0.45, xshift=19cm]
\draw[white, fill=dblue!15] (-5.5,1) -- (-5.5,5.5) -- (5.5,5.5) -- (5.5,1);
\draw[white, fill=dgreen!15] (-5.5,-1) -- (-5.5,-5.5) -- (5.5,-5.5) -- (5.5,-1);
\draw[white, thick] (-5.5,-5.5) grid (5.5,5.5);
\draw[thick, Apricot] (-5.5,0) -- (5.5,0);
\draw[thick, dblue, ->] (0,3) -- (0,4);
\draw[thick, dblue, ->] (0,3) -- (0,2);
\draw[thick, dblue, ->] (0,3) -- (-1,3);
\draw[thick, dblue, ->] (0,3) -- (1,4);
\draw[thick, dgreen, ->] (0,-3) -- (0,-2);
\draw[thick, dgreen, ->] (0,-3) -- (0,-4);
\draw[thick, dgreen, ->] (0,-3) -- (-1,-3);
\draw[thick, dgreen, ->] (0,-3) -- (1,-3);
\draw[thick, dgreen, ->] (0,-3) -- (1,-2);
\draw[thick, dgreen, ->] (0,-3) -- (-1,-4);
\draw[thick, dgreen, ->] (0,-3) -- (1,-4);
\draw[thick, nred, ->] (0,0) -- (0,1);
\draw[thick, nred, ->] (0,0) -- (-1,-1);
\draw[thick, nred, ->] (0,0) -- (1,0);
\end{scope}

\end{tikzpicture}
   \caption{Left: an example of spatially inhomogeneous model studied in \cite{BoMuSi-15,Mu-19,BuKa-19}. Right: block inhomogeneous model in the plane}
  \label{fig:plane_split}
\end{figure}
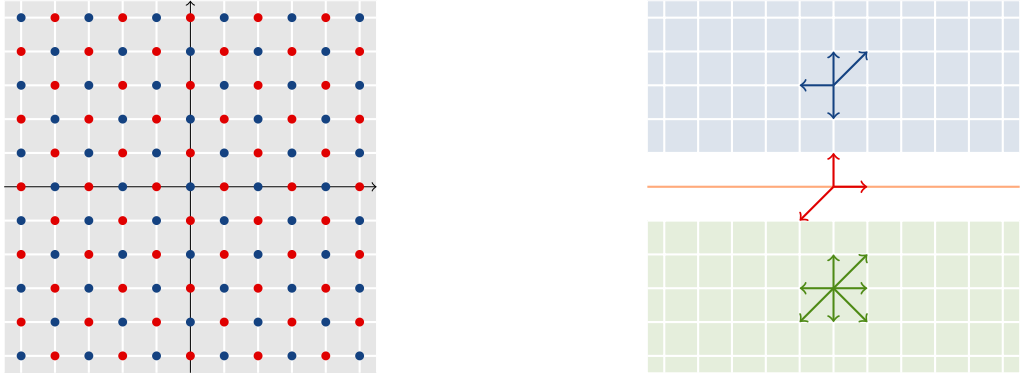
     
     Let us now present the JSQ model. This is a model with (say) two queues, in which (as its name suggests) the arriving customers choose the shortest queue; if the two queues happen to have the same length, then a queue is chosen according to an a priori fixed probability law. From a random walk viewpoint, this means splitting the quarter plane in two octants (cones of opening angle $\frac{\pi}{4}$) as on Figure \ref{fig:JSQ}. In general, the service times depend on the servers, and thus the transition probabilities are different in the upper and lower octants (one speaks about spatially inhomogeneous random walks, and of the general asymmetric JSQ). On the other hand, when the probability laws are symmetric in the diagonal, the model is said symmetric. Classical references are \cite{Fa-79,Ia-79,AdWeZi-91,FoMD-01,KuSu-03} and \cite[Chap.\ 10]{FaIaMa-17}. Surprisingly, the non-symmetric JSQ is still an open problem: a typical interesting problem in queueing theory would be to compute a closed-form expression for the stationary distribution.

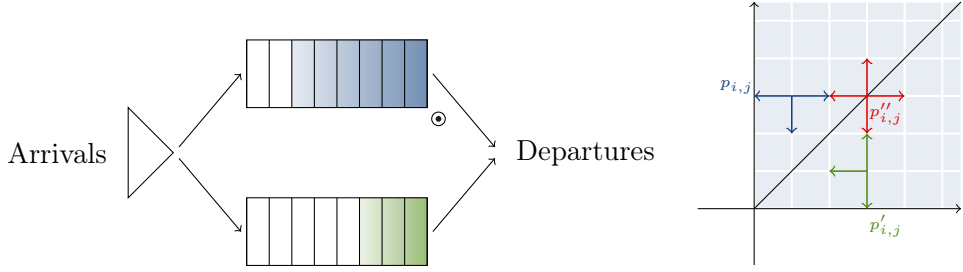
\begin{figure}[ht!]
\centering
\begin{tabular}{c c}
\begin{tikzpicture}
\begin{scope}[scale=0.3]
\draw (6.5,1.5) circle (0.3 cm);
\draw[thick, black!90, fill=black!90] (6.5,1.5) circle (0.08 cm);
\draw [black, left] (-7.75,0) node{Arrivals};
\draw [black, right] (9.5,0) node{Departures};
\draw (-5.25,0 ) -- ( -7.25,-2 ) -- (-7.25,2 ) -- cycle;
\shade[left color=dblue!10, right color=dblue!60] (0,2) rectangle (6,5);
\draw (-2,2) rectangle (6,5);
\foreach \x in {-2,-1,...,6} \draw (\x,2) -- (\x,5)  ;
\shade[left color=dgreen!10, right color=dgreen!60] (3,-2) rectangle (6,-5);
\draw (-2,-2) rectangle (6,-5);
\foreach \x in {-2,-1,...,6} \draw (\x,-2) -- (\x,-5)  ;
\draw [->] (-5,0.25) -- (-2.25,3.5);
\draw [->] (-5,-0.25) -- (-2.25,-3.5);
\draw [->] (6.25,3.5) -- (9,0.25);
\draw [->] (6.25,-3.5) -- (9,-0.25);
\end{scope}
\end{tikzpicture}
&
\begin{tikzpicture}
\begin{scope}[scale=0.5]
\draw[white, fill=dblue!10] (0,0) rectangle (5.5,5.5);
\draw[white, thick] (-1.5,-1.5) grid (5.5,5.5);
\draw[] (0,0) -- (5.5,5.5);
\draw[->] (0,-1.5) -- (0,5.5);
\draw[->] (-1.5,0) -- (5.5,0);
\draw[->, nred, semithick] (3,3) -- (4,3);
\draw[->, nred, semithick] (3,3) -- (2,3);
\draw[->, nred, semithick] (3,3) -- (3,4);
\draw[->, nred, semithick] (3,3) -- (3,2);
\node[below right, nred] at (2.8,3.1) {\tiny{$p''_{i,j}$}};
\draw[->,dgreen, semithick] (3,1) -- (3,2);
\draw[->,dgreen, semithick] (3,1) -- (2,1);
\draw[->,dgreen, semithick] (3,1) -- (3,0);
\node[below right, dgreen] at (2.8,0.1) {\tiny{$p'_{i,j}$}};
\draw[->, dblue, semithick] (1,3) -- (0,3);
\draw[->, dblue, semithick] (1,3) -- (1,2);
\draw[->, dblue, semithick] (1,3) -- (2,3);
\node[above left, dblue] at (0.25,2.8) {\tiny{$p_{i,j}$}};
\end{scope}
\end{tikzpicture}
\end{tabular}
\caption{Left: the JSQ model can be represented as a system of two queues, in which the customers choose the shortest one (the green one, on the picture). Right: representation of the JSQ model as an inhomogeneous random walk in the quadrant}
\label{fig:JSQ}
\end{figure}

Let us briefly notice that quadrant walks could also be treated with a JSQ approach, by decomposing the quarter plane into two octants as on Figure \ref{fig:JSQ}, see e.g.\ \cite{KuSu-03} for asymptotic results.

\paragraph{Main results: a contour-integral expression for the generating function.}
Throughout this paper we will do the following assumption:
\begin{enumerate}[label=(H),ref={\rm (H)}]
     \item\label{main_hyp}The step set $\mathcal S$ is symmetric (i.e., if $(i,j)\in\mathcal S$ then $(j,i)\in\mathcal S$) and does not contain the jumps $(-1,1)$ and $(1,-1)$.
\end{enumerate}
An exhaustive list of which step sets obey \ref{main_hyp} is given on Figures \ref{fig:symmetric_models_finite} and \ref{fig:symmetric_models_infinite}. We are not able to deal with asymmetric walks (as we are unable to solve the asymmetric JSQ model, see above), because of the complexity of the functional equations. The jumps $(-1,1)$ and $(1,-1)$ are discarded for similar reasons: they would lead to additional terms in the functional equation (see Figure \ref{fig:diag_bigstep}). 

\begin{figure}[htb]
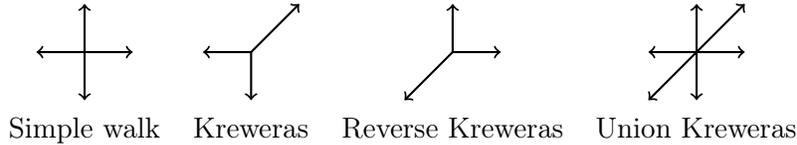
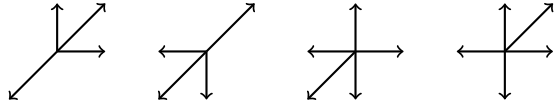

  \centering
  \begin{subfigure}[b]{1\textwidth}
        \centering 
  \begin{tabular}{ccccc}
    $\ \diagr{E,W,N,S}\ $ & $\ \diagr{NE,S,W}\ $  & $\ \diagr{N,E,SW}\ $ & $\ \diagr{NE,S,W,N,E,SW}\ $
\\
Simple walk&Kreweras&Reverse Kreweras & Union Kreweras
  \end{tabular}
   \caption{Symmetric models with a finite group. The notion of group associated to a model is recalled in Section \ref{sec:preliminaries}}
  \label{fig:symmetric_models_finite}
      \end{subfigure}

      \begin{subfigure}[b]{1\textwidth}
  \centering
  \begin{tabular}{cccc}
    $\ \diagr{E,NE,N,SW}\ $ &$\ \diagr{W,NE,S,SW}\ $ & $\ \diagr{E,W,N,S,SW}\ $& $\ \diagr{E,W,N,S,NE}\ $ 
  \end{tabular}
   \caption{Symmetric models with infinite group}
  \label{fig:symmetric_models_infinite}
    \end{subfigure}
     \caption{Symmetric models with no jumps $(-1,1)$ and $(1,-1)$}\label{fig:symmetric_models}
\end{figure}

Our main result is a contour-integral expression for the diagonal section
\begin{equation*}
     D(x)=\sum_{n\geq 0,i\geq 0}c_{i,i}(n)x^it^n.
\end{equation*}
We shall see later that knowing $D(x)$ actually suffices to give a complete solution to the problem (i.e., to find an expression for $C(x,y)$ in \eqref{eq:generating_function_3/4}). Let us postpone to Theorem \ref{thm:first_main} the very precise statement, and instead let us give now the main idea and the shape of the solution. We will show that
\begin{equation}
\label{eq:presentation_main_theorem}
	D(x)=w'(x)f(w(x))\int g(u,w(z)) \frac{w'(z)}{w(x)-w(z)}\dz,
\end{equation}
where $f$ and $g$ are algebraic functions. The integral in \eqref{eq:presentation_main_theorem} is taken over a quartic curve, constructed from the step set of the model. The function $w$ is interpreted as a conformal mapping for the domain bounded by the quartic, and its algebraic nature heavily depends on the model under consideration: it can be algebraic (finite group case) or non-D-finite (otherwise).

\paragraph{Five consequences of our main results.} Our first contribution is about methodology: we show that under the symmetry condition \ref{main_hyp}, three-quadrant walk models are exactly solvable, in the sense that their generating function admits an explicit (contour-integral) expression \eqref{eq:presentation_main_theorem}.

\medskip

The second point is that our techniques allow to compare walks in a quadrant and walks in three quadrants. More precisely, it is proved in \cite{Ra-12} that the generating function counting quadrant walks ending on the horizontal axis can typically be expressed as
\begin{equation}
\label{eq:presentation_1/4}
     \widetilde f(x)\int \widetilde g(z)\frac{w'(z)}{w(x)-w(z)}\dz,
\end{equation}
with the same function $w$ as in \eqref{eq:presentation_main_theorem} but different functions $\widetilde f$ (rational) and $\widetilde g$ (algebraic). Though simpler, Equation~\eqref{eq:presentation_1/4} is quite similar to \eqref{eq:presentation_main_theorem}. This similarity opens the way to prove combinatorial formulas relating the two models.

\medskip

Our third corollary is a partial answer to two questions raised by Bousquet-M\'elou in \cite{BM-16}, that we briefly recall: first, could it be that for any step set associated with a finite group, the generating function $C(x, y)$ is D-finite? Second, could it be that for the four step sets [Kreweras, reverse Kreweras, union Kreweras (see Figure \ref{fig:symmetric_models_finite}) and Gessel (Figure \ref{fig:some_models})], for which [the quadrant generating function] is known to be algebraic, $C(x, y)$ is also algebraic? 

The expression \eqref{eq:presentation_main_theorem} rather easily entails that if $w$ is algebraic (which will correspond to the finite group case, see Section \ref{sec:preliminaries}), the generating function $D(x)$ is D-finite, being the Cauchy integral of an algebraic function. On the other hand, when the group is infinite the function $w$ is non-D-finite by \cite[Thm.\ 2]{Ra-12}, and the expression \eqref{eq:presentation_main_theorem} uses non-D-finite functions (note, this does not a priori imply that $D(x)$ itself is non-D-finite, but does provide some evidence). 

\medskip

Next, although we do not solve them, the expression \eqref{eq:presentation_main_theorem} provides a way to attack the following questions:
\begin{itemize}
     \item Starting from the integral \eqref{eq:presentation_1/4}, various asymptotic questions concerning quadrant models are solved in \cite{FaRa-12} (asymptotics of the excursions, of the number of walks returning to one axis, etc.). Similar arguments should lead to the asymptotics of walks in three quadrants. Remember, however, that the asymptotics of the excursion sequence is already found in \cite{Mu-19}.
     \item A further natural question (still unsolved in the quadrant case) is to find, in the finite group case, a concrete differential equation (or minimal polynomial in case of algebraicity) for the generating function, starting from the contour integrals \eqref{eq:presentation_main_theorem} or \eqref{eq:presentation_1/4}. It seems that the technique of creative telescoping could be applied to the contour integral expressions.
     \item Several interesting (and sometimes surprising) combinatorial identities relating quadrant walks to three-quadrant walks are proved in \cite{BM-16} (in particular, a proof of the former Gessel's conjecture by means of simple walks in $\mathcal C$ and the reflection principle). Moreover, Bousquet-M\'elou asks in \cite{BM-16} whether $C(x,y)$ could differ from (a simple D-finite series related to) the quadrant generating function by an algebraic series? Taking advantage of the similarity between \eqref{eq:presentation_main_theorem} and \eqref{eq:presentation_1/4} provides a starting point to answer this question.
\end{itemize}

Finally, along the way of proving our results, we develop a noteworthy concept of anti-Tutte's invariant, namely a function $g$ such that ($\overline{y}$ denoting the complex conjugate number of $y\in\mathbb C$)
\begin{equation}
\label{eq:anti_Tutte}
     g(y)=\frac{1}{g(\overline{y})}
\end{equation}
when $y$ lies on the contour of \eqref{eq:presentation_main_theorem}. The terminology comes from \cite{BeBMRa-17}, where a function $g$ satisfying to $g(y)=g(\overline{y})$ is interpreted as a Tutte invariant and is strongly used in solving the models. Originally, Tutte introduced the notion of invariant to solve a functional equation counting colored planar triangulations, see \cite{Tu-95}.  Tutte's equation is rather close to functional equations arising in two-dimensional counting problems. Interestingly, a function $g$ as in \eqref{eq:anti_Tutte} appears in the book \cite{CoBo-83}, which proposes an analytic approach to quadrant walk problems (the latter is more general than \cite{FaIaMa-17} in the sense that it works for arbitrarily large positive jumps, i.e., not only small steps). In \cite{CoBo-83} it is further assumed that $g(\overline{y})=\overline{g(y)}$, so that with \eqref{eq:anti_Tutte} one has $\vert g(y)\vert=1$, and $g$ may be interpreted as a conformal mapping from the domain bounded by contour of \eqref{eq:presentation_main_theorem} onto the unit disc.

\paragraph{Equations with (too) many unknowns.}
What about non-symmetric models? From a functional equation viewpoint, the latter are close to random walks with big jumps \cite{FaRa-15,BoBMMe-18} or random walks with catastrophes \cite{BaWa-17}, in the sense that the functional equation has more than two unknowns in its right-hand side. One idea to get rid of these extra terms is to transform the initial functional equation, as in \cite{BM-16}, where Bousquet-M\'elou solves the simple and diagonal models, starting from non-symmetric points ($(-1,0)$, for instance). Another idea, present in \cite{BoBMMe-18}, is to extend the kernel method by computing weighted sums of several functional equations, each of them being an algebraic substitution of the initial equation. However, finding such combinations is very difficult in general. 

From the complex analysis counterpart \cite{FaIaMa-17,Ra-12,FaRa-15}, equations with many unknowns become systems of boundary value problems, which seem not to have a solution in the literature. It is also shown in \cite[Chap.\ 10]{FaIaMa-17} that the asymmetric JSQ is equivalent to solving an integral Fredholm equation for the generating function, but again, no closed-form expression seems to exist.

\paragraph{A conjecture.} 
Although it is not directly inspired by our work, let us state the following. Consider an arbitrary finite group step set $\mathcal S$ (not necessarily satisfying to \ref{main_hyp} but with small steps). We conjecture that the generating function for walks in three quadrants $C(x,y)$ is algebraic as soon as the starting point $(i_0,j_0)\in\mathcal C$ is such that $i_0=-1$ or $j_0=-1$.

This conjecture is motivated by an analogy with the quarter plane, in which the following result holds: a finite group model (having group $G$) with starting point at $(i,j)$ is algebraic if and only if the orbit-sum 
     \begin{equation*}
     \sum_{g\in G}(-1)^gg(x^{i+1}y^{j+1})
     \end{equation*}
is identically zero, see \cite{BMMi-10,BoKa-10,KuRa-12}. Taking $i=-1$ in the sum above (which obviously is not possible in the quadrant case!) yields a zero orbit-sum---more generally, the orbit-sum of any function depending on only one of the two variables $x,y$ is zero.


\paragraph{Structure of the paper.}$ $

$\bullet$ Section \ref{sec:preliminaries}: statement of various functional equations satisfied by the generating functions (in particular Lemma \ref{lem:functional_equation_sym}), definition of the group of the model, study of the zero-set of the kernel

$\bullet$ Section \ref{sec:expression_GF}: statement of a boundary value problem (BVP) satisfied by the diagonal generating function (Lemma \ref{lem:BVP}), resolution of the BVP (Theorems \ref{thm:first_main} and \ref{thm:second_main})

$\bullet$ Appendix \ref{app:expression_Gluing}: list and properties of conformal mappings used in Theorems \ref{thm:first_main} and \ref{thm:second_main}

$\bullet$ Appendix \ref{app:RHBVP}: important statements from the theory of BVP

$\bullet$ Appendix \ref{app:proof_lem1}: proof of the main functional equation stated in Lemma \ref{lem:functional_equation_sym}

\paragraph{Acknowledgments.} We are most grateful to Marni Mishna for her constant support and many enlightening discussions. We would like to also thank Alin Bostan, Irina Ignatiouk-Robert, Sami Mustapha and Michael Wallner for various discussions. Finally, we thank an anonymous referee for his/her numerous suggestions, which in particular led us to obtain series expansions of the contour integrals given in our main theorem.

\section{Preliminaries}
\label{sec:preliminaries}

\subsection{Kernel functional equations}
\label{subsec:kernel_functional_equations}

The starting point is to write a functional equation satisfied by the generating function \eqref{eq:generating_function_3/4}, which, as explained in the introduction, translates the step-by-step construction of a walk. Before dealing with this functional equation, let us define some important objects. 

First of all, a step set $\mathcal S\subset \{-1,0,1\}^2$ is characterized by its inventory (or jump polynomial) $\sum_{(i,j)\in\mathcal{S}}x^i y^j$ as well as by the associated kernel
\begin{equation}
\label{eq:kernel}
     K(x,y)=xy \Bigg(t\sum_{(i,j)\in\mathcal{S}}x^i y^j -1\Bigg).
\end{equation}
The kernel is a polynomial of degree two in $x$ and $y$, which we can write as
\begin{equation}
\label{eq:kernel_expanded}
     K(x,y)=\widetilde{a}(y)x^2+\widetilde{b}(y)x+\widetilde{c}(y) = a(x)y^2+b(x)y+c(x),
\end{equation}
where
\begin{equation}
\label{eq:coeff_kernel}
\left\{
\begin{array}{l l l}
a(x)=tx\sum_{(i,1)\in\mathcal{S}}x^i, & b(x)=tx\sum_{(i,0)\in\mathcal{S}}x^i -x, & c(x)=tx\sum_{(i,-1)\in\mathcal{S}}x^i,\\
\widetilde{a}(y)=ty\sum_{(1,j)\in\mathcal{S}}y^j, & \widetilde{b}(y)=ty\sum_{(0,j)\in\mathcal{S}}y^j -y, & \widetilde{c}(y)=ty\sum_{(-1,j)\in\mathcal{S}}y^j.
\end{array}
\right.
\end{equation}
Define further $\delta_{-1,-1}=1$ if $(-1,-1)\in\mathcal S$ and $\delta_{-1,-1}=0$ otherwise. In the three-quarter plane, we can generalize Equation~(12) in \cite[Sec.\ 2.1]{BM-16} and deduce the following equation satisfied by $C(x,y)$ defined in \eqref{eq:generating_function_3/4}:
\begin{equation}
\label{eq:functional_equation_3/4}
     K(x,y)C(x,y)=c(x)C_{-0}(x^{-1})+\widetilde{c}(y)C_{0-}(y^{-1})-t\delta_{-1,-1} C_{0,0}-xy,
\end{equation}
where
\begin{equation*}
     C_{-0}(x^{-1})= \sum_{n\geq0,i\leq0} c_{i,0}(n)x^{i}t^{n},\quad
     C_{0-}(y^{-1}) = \sum_{n\geq0,j\leq0} c_{0,j}(n)y^{j}t^{n}\quad\text{and}\quad
     C_{0,0}=\sum_{n\geq0}c_{0,0}(n)t^{n}.
\end{equation*}
In comparison, let us recall the standard functional equation in the case of the quarter plane
\begin{equation*}
     \mathcal Q=\{(i,j)\in\mathbb Z^2: i\geq0 \text{ and } j\geq0\}.
\end{equation*}
By \cite[Lem.\ 4]{BMMi-10} and using similar notation as above, the generating function
\begin{equation}
\label{eq:generating_function_1/4}
     Q(x,y)=\sum_{n\geq0}\sum_{(i,j)\in\mathcal Q}q_{i,j}(n)x^{i}y^{j} t^{n}
\end{equation}
satisfies the equation
\begin{equation}
\label{eq:functional_equation_1/4}
     K(x,y)Q(x,y)=c(x)Q_{-0}(x)+\widetilde{c}(y)Q_{0-}(y)-t\delta_{-1,-1} Q_{0,0}-xy,
\end{equation}
where
\begin{multline}
\label{eq:sections_1/4}
     Q_{-0}(x)= \sum_{n\geq0,i\geq 0} q_{i,0}(n)x^{i}t^{n},\quad
     Q_{0-}(y)=\sum_{n\geq0,j\geq 0} q_{0,j}(n)y^{j}t^{n}\quad\text{and}\quad
     Q_{0,0}=\sum_{n\geq0}q_{0,0}(n)t^{n}.
\end{multline}
At first sight, the two functional equations \eqref{eq:functional_equation_3/4} and \eqref{eq:functional_equation_1/4} are very similar. However, due to the presence of infinitely many terms with positive and negative valuations in $x$ or $y$, the first one is much more complicated, and almost all the methodology of \cite{BMMi-10,Ra-12} (namely, performing algebraic substitutions or evaluating the functional equation at well-chosen complex points) breaks down, as the series are no longer convergent. 

The idea in \cite{BM-16} is to see $\mathcal C$ as the union of three quarter planes, and to state for each quadrant a new equation, which is more complicated but (by construction) may be evaluated. Our strategy follows the same line: we split the three-quadrant cone in two domains (two cones of opening angle $\frac{3\pi}{4}$, see Figure \ref{fig:three-quarter_split}) and write two functional equations, one for each domain.

\subsection{Functional equations for the $\frac{3\pi}{4}$-cone walks}

In this section and in the remainder of our paper, we shall use two different step sets, $\widehat{\mathcal S}$ and $\mathcal S$. The first one, $\widehat{\mathcal S}$, will refer to the main step set, corresponding to the walks in the three-quarter plane we are counting. Accordingly, we will rename all quantities associated to the main step set with a hat, for instance the kernel will be denoted by $\widehat{K}(x,y)$. The second step set, $\mathcal S$, is associated to $\widehat{\mathcal S}$ after the change of variable \eqref{eq:change_var}. Quantities with no hat will be associated to the step set $\mathcal S$, for instance the kernel $K(x,y)$. In order not to make the notation to heavy and because in this case there is no possible ambiguity, the only exception to this rule will be the coefficients $c_{i,j}(n)$ (with no hat), which will always correspond to $\widehat{\mathcal S}$.

Having said that, we start by splitting the domain of possible ends of the walks into three parts: the diagonal, the lower part $\{i\geq 0, j\leq i-1\}$ and the upper part $\{j\geq 0, i\leq j-1\}$, see Figure \ref{fig:three-quarter_split}. We may write
\begin{equation}
\label{eq:equation_cut3parts}
     C(x,y)=\widehat{L}(x,y)+\widehat{D}(x,y)+\widehat{U}(x,y),
\end{equation}
where 
\begin{equation*}
\widehat{L}(x,y)=\sum\limits_{\substack{i \geq 0 \\  j\leq i-1 \\n\geq 0}}c_{i,j}(n)x^i y^j t^n, \quad
\widehat{D}(x,y)=\sum\limits_{\substack{i\geq 0 \\ n\geq 0}}c_{i,i}(n)x^i y^i t^n \quad\text{and}\quad
\widehat{U}(x,y)=\sum\limits_{\substack{j \geq 0 \\ i\leq j-1 \\n\geq 0}}c_{i,j}(n)x^i y^j t^n.
\end{equation*}
Let $\delta_{i,j}=1$ if $(i,j)\in\mathcal S$ and $0$ otherwise.

\begin{figure}[htb]
\centering
\begin{tikzpicture}
\begin{scope}[scale=0.45]
\tikzstyle{quadri}=[rectangle,draw,fill=white]
\draw[white, fill=dblue!15] (0,-1) -- (5.5,4.5) -- (5.5,-5.5) -- (0,-5.5);
\draw[white, fill=dgreen!15] (-1,0) -- (-5.5,0) -- (-5.5,5.5) -- (4.5,5.5);
\draw[white, thick] (0,-5.5) grid (5.5,5.5);
\draw[white, thick] (-5.5,0) grid (0,5.5);
\draw[Apricot!90, thick] (0,0) -- (5.5,5.5);
\draw[dashed, dgreen!90, thick] (-1,0) -- (4.5,5.5);
\draw[dashed, dblue!90, thick] (0,-1) -- (5.5,4.5);
\draw[->] (0,-5.5) -- (0,5.5);
\draw[->] (-5.5,0) -- (5.5,0);
\node[dblue!90,quadri] at (3,-3) {$\widehat{L}(x,y)$};
\node[dgreen!90, quadri] at (-3,3) {$\widehat{U}(x,y)$};
\node[Apricot!100, quadri] at (7.4,6.5) {$\widehat{D}(x,y)$};
\node[dgreen!100, quadri, dashed] at (2.8,6.5) {$\widehat{D}^u(x,y)$};
\node[dblue!100, quadri, dashed] at (7.6,4) {$\widehat{D}^\ell(x,y)$};
\end{scope}
\end{tikzpicture}
\caption{Decomposition of the three-quarter plane and associated generating functions}
\label{fig:some_sections}
\end{figure}
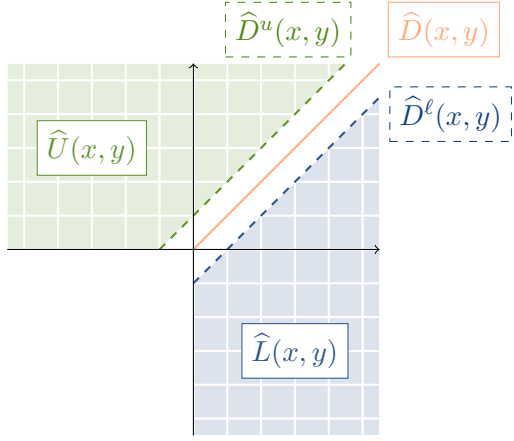

\begin{lem}
\label{lem:functional_equation_sym}
For any step set which satisfies \ref{main_hyp} and starts at $(0,0)$, one has
\begin{multline}
\label{eq:functional_equation_sym}
\widehat{K}(x,y)\widehat{L}(x,y)=-\frac{1}{2}xy+txy\left(\delta_{-1,-1}x^{-1}y^{-1}+ \delta_{-1,0}x^{-1} \right)\widehat{L}_{0-}(y^{-1}) + \frac{1}{2}t\delta_{-1,-1}\widehat{D}(0,0)
\\
-xy\left(-\frac{1}{2}+t\left(\frac{1}{2}(\delta_{1,1}xy+\delta_{-1,-1}x^{-1}y^{-1})+ \delta_{0,-1}y^{-1} + \delta_{1,0}x \right) \right)\widehat{D}(x,y),
\end{multline}
with $\widehat{L}_{-0}(y^{-1})= \sum_{n\geq0,j<0} c_{0,j}(n)y^{j}t^{n}$.
\end{lem}
The proof of Lemma \ref{lem:functional_equation_sym} is postponed to Appendix \ref{app:proof_lem1}, as it is elementary but a bit technical. The functional equation for non-symmetric models (as well as for symmetric models with non-diagonal starting points) is commented in Appendix~\ref{app:proof_lem1}. Here we will only consider symmetric models starting at $(0,0)$, but notice that our study can be easily generalized to arbitrary diagonal starting points. 

In order to simplify the functional equation \eqref{eq:functional_equation_sym}, we perform the change of variable
\begin{equation}
\label{eq:change_var}
     \varphi(x,y)=(xy, x^{-1}).
\end{equation}
Then \eqref{eq:functional_equation_sym} becomes 
\begin{equation}
\label{eq:functional_equation_octant}
K(x,y)L(x,y)=c(x)L_{-0}(x)-x\left(x\widetilde{a}(y)+\frac{\widetilde{b}(y)}{2}\right)D(y)+\frac{1}{2}txD(0)-\frac{1}{2}xy,
\end{equation}
where $K(x,y) =xy(t\sum_{(i,j)\in \mathcal{S}}x^{i-j}y^{i}-1)=x\widehat{K}(\varphi(x,y))$, $L_{-0}(x)=\sum_{n\geq 0,j\geq 1} c_{0,-j}x^j t^n$ and similarly
\begin{equation}
\label{eq:def_diagonal}
     L(x,y) = \widehat{L}(\varphi(x,y)) = \sum\limits_{\substack{i \geq 1 \\  j\geq 0 \\n\geq 0}}c_{j,j-i}(n)x^i y^j t^n 
     \quad \text{and}\quad 
	D(y) = \widehat{D}(\varphi(x,y)) = \sum\limits_{\substack{i\geq 0 \\ n\geq 0}} c_{i,i}(n)y^i t^n.
\end{equation}

The change of coordinates $\varphi$ simplifies the resolution of the problem, as the functional equation \eqref{eq:functional_equation_octant} is closer to a (solvable) quadrant equation; compare with \eqref{eq:functional_equation_1/4}. Throughout the manuscript, functions with (resp.\ without) a hat will be associated to the step set $\mathcal S$ (resp.\ to the step set after change of variable $\varphi$). We have 
\begin{equation*}
     \mathcal{S}=\varphi(\widehat{\mathcal{S}})=\{(i-j,i):(i,j)\in\widehat{\mathcal{S}}\}.
\end{equation*}
For the reader's convenience, we have represented on Table \ref{tab:transform_models} the effect of $\varphi$ on the symmetric models of Figures \ref{fig:symmetric_models_finite} and \ref{fig:symmetric_models_infinite}. We also remark on Figure \ref{fig:diag_bigstep} that the presence of anti-diagonal jumps $(-1,1)$ or $(1,-1)$ would lead to the bigger steps $(-2,-1)$ or $(2,1)$: this is the reason why they are discarded.

\subsection{Group of the model}

In this part and in Section \ref{subsec:Roots_of_the_kernel} as well, we remove the hat from our notation: indeed, the statements are valid for all step sets (with or without hat).

With our notation \eqref{eq:coeff_kernel}, the group of the walk is the dihedral group of bi-rational transformations $\langle\Phi,\Psi\rangle$ generated by the involutions
\begin{equation*}
\Phi(x,y)=\left( \frac{\widetilde{c}(y)}{\widetilde{a}(y)}\frac{1}{x},y\right) \qquad \text{and}\qquad
\Psi(x,y)=\left(x, \frac{c(x)}{a(x)}\frac{1}{y}\right).
\end{equation*}
It was introduced in \cite{Ma-72} in a probabilistic context and further used in \cite{FaIaMa-17,BMMi-10}. The group $\langle\Phi,\Psi\rangle$ may be finite (of even order, larger than or equal to $\geq4$) or infinite. The order of the group for the $79$ non-equivalent quadrant models is computed in \cite{BMMi-10}: there are $23$ models with a finite group ($16$ of order $4$, $5$ of order $6$ and $2$ of order $8$) and $56$ models with infinite order.  

For instance, the simple walk has a group of order $4$, while the three right models on Figure \ref{fig:symmetric_models_finite} have a group of order $6$. Indeed, taking Kreweras model as an example, we have $\Phi(x,y)=(\frac{1}{xy},y)$ and $\Psi(x,y)=(x,\frac{1}{xy})$, and the orbit of $(x,y)$ under the action of $\Phi$ and $\Psi$ is
\begin{equation*}
(x,y)
\overset{\Phi}{\rightarrow}
(\textstyle{\frac{1}{xy}},y)
\overset{\Psi}{\rightarrow}
(\textstyle{\frac{1}{xy}},x)
\overset{\Phi}{\rightarrow}
(y,x)
\overset{\Psi}{\rightarrow}
(y,\textstyle{\frac{1}{xy}})
\overset{\Phi}{\rightarrow}
(x,\textstyle{\frac{1}{xy}})
\overset{\Psi}{\rightarrow}
(x,y).
\end{equation*}

\begin{table}[htb]
    \begin{minipage}[c]{.46\linewidth}
        \centering
       \begin{tabular}{|c | c |}
\hline
Model & Image under $\varphi$ \\
\hline \hline
$\ \diagr{N,S,E,W}\ $ & $\ \diagr{E,W,NE,SW}\ $ \\\hline
$\ \diagr{S,W,NE}\ $ & $\ \diagr{E,N,SW}\ $ \\\hline
$\ \diagr{E,N,SW}\ $ & $\ \diagr{S,W,NE}\ $ \\\hline
$\ \diagr{E, NE, N, W, SW, S}\ $ & $\ \diagr{E, NE, N, W, SW, S}\ $ \\\hline
\end{tabular}
    \end{minipage}
    \hfill%
    \begin{minipage}[c]{.46\linewidth}
        \centering
        \begin{tabular}{|c | c |}
\hline
Model & Image under $\varphi$ \\
\hline \hline
$\ \diagr{E, NE, N, SW}\ $ &  $\ \diagr{NE, N, W, S}\ $ \\\hline
$\ \diagr{NE, W, SW, S}\ $ &$\ \diagr{E, N, SW, S}\ $ \\\hline
$\ \diagr{E, N, W, SW, S}\ $ & $\ \diagr{E, NE, W, SW, S}\ $ \\\hline
$\ \diagr{E, NE, N, W, S}\ $ & $\ \diagr{E, NE, N, W, SW}\ $ \\\hline
        \end{tabular}
    \end{minipage}
\caption{Transformation $\varphi$ on the eight symmetric models (with finite group on the left and infinite group on the right) without the steps $(-1,1)$ and $(1,-1)$. In particular, the simple walk is related by $\varphi$ to Gessel's model. After \cite{BM-16}, this is another illustration that counting simple walks in three-quarter plane is related to counting Gessel walks in a quadrant}
  \label{tab:transform_models}
\end{table}

\begin{figure}[htb]
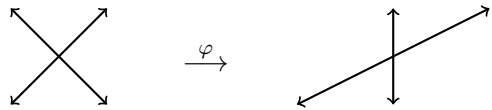

\centering
\begin{tabular}{m{2cm} m{1cm} m{2cm}}
$\ \diagr{NE,SE,NW,SW}\ $ & $\stackrel{\varphi}{\longrightarrow}$  & $\ \diagr{N, S, S2W, N2E}\ $
\end{tabular}
\caption{The diagonal model is transformed by $\varphi$ into a model with bigger steps}
\label{fig:diag_bigstep}
\end{figure}

\subsection{Roots and curves defined by the kernel}
\label{subsec:Roots_of_the_kernel}

We define the discriminants in $x$ and $y$ of the kernel \eqref{eq:kernel_expanded}:
\begin{equation}
\label{eq:discriminants}
     \widetilde{d}(y)=\widetilde{b}(y)^2-4\widetilde{a}(y)\widetilde{c}(y)\qquad \text{and}\qquad d(x)=b(x)^2-4a(x)c(x).
\end{equation}
The discriminant $d(x)$ (resp.\ $\widetilde{d}(y)$) in \eqref{eq:discriminants} is a polynomial of degree three or four. Hence it admits four roots (also called branch points) $x_1, x_2, x_3, x_4$ (resp.\ $y_1, y_2, y_3, y_4$), with $x_4=\infty$ (resp.\ $y_4=\infty$) when $d(x)$ (resp.\ $\widetilde{d}(y)$) is of degree $3$. 
\begin{lem}[Sec.~3.2 in \cite{Ra-12}]
\label{lem:properties_branch_points}
Let $t\in(0,1/\vert\mathcal S\vert)$.
The branch points $x_i$ are real and distinct. Two of them (say $x_1$ and $x_2$) are in the open unit disc, with $x_1 < x_2$ and $x_2>0$. The other two  (say $x_3$ and $x_4$) are outside the closed unit disc, with $x_3>0$ and  $x_3 <  x_4$ if $x_4>0$. The discriminant $d(x)$ is negative on $(x_1,x_2)$ and $(x_3,x_4)$, where if $x_4<0$, the set $(x_3,x_4)$ stands for the union of intervals $(x_3,\infty)\cup(-\infty,x_4)$. Symmetric results hold for the branch points $y_i$.
\end{lem}
Let $Y(x)$ (resp.\ $X(y)$) be the algebraic function defined by the relation $K(x,Y(x))=0$ (resp.\ $K(X(y),y)=0$). Obviously with \eqref{eq:kernel_expanded} and \eqref{eq:discriminants} we have
\begin{equation}
\label{eq:algebraic_expressions_Y_X}
     Y(x)=\frac{-b(x)\pm\sqrt{d(x)}}{2a(x)}\qquad \text{and}\qquad X(y)=\frac{-\widetilde{b}(y)\pm\sqrt{\widetilde{d}(y)}}{2\widetilde{a}(y)}.
\end{equation}
The function $Y$ has two branches $Y_0$ and $Y_1$, which are meromorphic on the cut plane $\mathbb{C} \setminus([x_1, x_2] \cup [x_3, x_4])$. On the cuts $[x_1,x_2]$ and $[x_3,x_4]$, the two branches still exist and are complex conjugate (but possibly infinite at $x_1=0$, as discussed in Lemma \ref{lem:properties_curves}). At the branch points $x_i$, we have $Y_0(x_i)=Y_1(x_i)$ (when finite), and we denote this common value by $Y(x_i)$. 
 
Fix the notation of the branches by choosing $Y_0=Y_-$ and $Y_1=Y_+$ in \eqref{eq:algebraic_expressions_Y_X}. We further fix the determination of the logarithm so as to have $\sqrt{d(x)}>0$ on $(x_2,x_3)$. Then clearly with \eqref{eq:algebraic_expressions_Y_X} we have 
\begin{equation}
\label{eq:global_inequality}
     \vert Y_0\vert \leq \vert Y_1\vert
\end{equation}
on $(x_2,x_3)$, and as proved in \cite[Thm.~5.3.3]{FaIaMa-17}, the inequality \eqref{eq:global_inequality} holds true on the whole complex plane and is strict, except on the cuts, where $Y_0$ and $Y_1$ are complex conjugate.   
 
A key object is the curve $\mathcal L$ defined by
\begin{equation}
\label{eq:curve_L}
     \mathcal L =Y_0([x_1,x_2])\cup Y_1([x_1,x_2])=\{y\in \mathbb C:
     K(x,y)=0 \text{ and } x\in[x_1,x_2]\}. 
\end{equation}
By construction, it is symmetric with respect to the real axis. We denote by $\mathcal G_\mathcal L$ the open domain delimited by $\mathcal L$ and avoiding  the real point at $+\infty$. See Figures~\ref{fig:curves} and \ref{fig:curve_Kreweras} for a few examples. Furthermore, let $\mathcal{L}_0$ (resp.\ $\mathcal{L}_1$) be the upper (resp.\ lower) half of $\mathcal{L}$, i.e., the part of $\mathcal L$ with non-negative (resp.\ non-positive) imaginary part, see Figure \ref{fig:conformal-functions}. Likewise, we define $\mathcal M =X_0([y_1,y_2])\cup X_1([y_1,y_2])$.

\begin{lem}[Lem.~18 in \cite{BeBMRa-17}]
\label{lem:properties_curves}
The curve $\mathcal L$ in \eqref{eq:curve_L} is symmetric in the real axis. It intersects this axis at $Y(x_2)>0$. 

If $\mathcal L$ is unbounded, $Y(x_2)$ is the only intersection point. This occurs if  and only if neither $(-1,1)$ nor $(-1,0)$ belong to $\mathcal S$. In this case, $x_1=0$ and the only point of $[x_1, x_2]$ where at least one branch $Y_i(x)$ is infinite is $x_1$ (and then both branches are infinite there). Otherwise, the curve $\mathcal L$ goes through a second real point, namely $Y(x_1) \leq 0$. 

Consequently, the point $0$ is either in the domain $\mathcal G_\mathcal L$ or on the curve $\mathcal L$. The domain $\mathcal G_\mathcal L$ also contains the (real) branch points $y_1$ and $y_2$, of modulus less than $1$. The other two branch points, $y_3$ and $y_4$, are in the complement of $\mathcal G_\mathcal L \cup\mathcal L$.
\end{lem} 

\begin{figure}[htb]
\centering
\begin{tabular}{cccc}
\begin{tikzpicture}
\begin{scope}[scale=0.2]
\draw [dblue!90, thick] plot[smooth] file{./figures/Y0.dat};
\draw [dblue!90, thick] plot[smooth] file{./figures/Y1.dat};
\draw [->] (0,-10) -- (0,10);
\draw [->] (-12,0) -- (13,0);
\draw[thick, dblue!90, fill=dblue!90] (0,0) circle (0.15 cm);
\draw [dblue!90] (-0.9,-1.2) node{\tiny{$y_1$}};
\draw[thick, dblue!90, fill=dblue!90] (0.5,0) circle (0.15 cm);
\draw [dblue!90] (0.9,-1.2) node{\tiny{$y_2$}};
\draw[thick, dblue!90, fill=dblue!90] (12,0) circle (0.15 cm);
\draw [dblue!90] (12.2,-1.2) node{\tiny{$y_3$}};
\end{scope}
\end{tikzpicture}
\qquad
&
\qquad
\begin{tikzpicture}
\begin{scope}[scale=0.2]
\draw [dblue!90, thick] plot[smooth, xshift= -290.5cm, xscale=100] file{./figures/Y1infini2.dat};
\draw [dblue!90,thick] plot[smooth, xshift= -290.5cm, xscale=100 ] file{./figures/Y0infini2.dat};
\draw [->] (0,-10) -- (0,10);
\draw [->] (-8,0) -- (17,0);
\draw[thick, dblue!90, fill=dblue!90] (0.1,0) circle (0.15 cm);
\draw [dblue!90] (-0.6,-0.8) node{\tiny{$y_1$}};
\draw[thick, dblue!90, fill=dblue!90] (0.8,0) circle (0.15 cm);
\draw [dblue!90] (1.1,-0.8) node{\tiny{$y_2$}};
\draw[thick, dblue!90, fill=dblue!90] (7,0) circle (0.15 cm);
\draw [dblue!90] (7,-0.8) node{\tiny{$y_3$}};
\draw[thick, dblue!90, fill=dblue!90] (9,0) circle (0.15 cm);
\draw [dblue!90] (9,-0.8) node{\tiny{$y_4$}};		
\end{scope}
\end{tikzpicture}
  \end{tabular}
   \caption{The curve $\mathcal{L}$ for Gessel's model (on the left) and for the model with jumps $\{{\sf E},{\sf N},{\sf SW},{\sf S}\}$ (on the right), for $t=1/8$}
\label{fig:curves}
\end{figure}
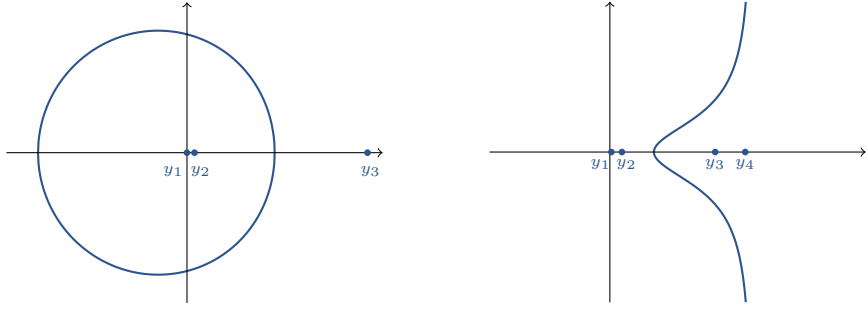

The step sets with jumps $\{{\sf E},{\sf N},{\sf SW}\}$, $\{{\sf E},{\sf NE},{\sf N},{\sf S}\}$ and $\{{\sf E},{\sf N},{\sf SW},{\sf S}\}$ have an unbounded contour, whereas the other models in Table \ref{tab:transform_models} have a bounded contour. 

We close this section by introducing a particular conformal mapping for $\mathcal{G}_\mathcal L$, which will happen to be very useful for our study.

\begin{defn}[Conformal gluing function]
\label{defn:CGF}
A function $w$ is said to be a conformal gluing function for the set $\mathcal{G}_\mathcal L$ if:
\begin{itemize}
     \item $w$ is meromorphic in $\mathcal{G}_\mathcal L$ and admits finite limits on $\mathcal L$;
     \item $w$ is one-to-one on $\mathcal{G}_\mathcal L$;
     \item for all $y$ on $\mathcal L$, $w(y)=w(\overline{y})$.
\end{itemize}
\end{defn}

For example, $w(y)=\frac{1}{2}(y+\frac{1}{y})$ is a conformal gluing function for the unit disc. See Appendix~\ref{app:expression_Gluing} for a list of conformal gluing functions associated to the domains we are considering.

\section{Expression for the generating functions}
\label{sec:expression_GF}

\subsection{Main results and discussion}

The first and crucial point is to prove that the diagonal $D(y)$ in \eqref{eq:def_diagonal} satisfies a boundary value problem (BVP), in the sense of the lemma below, the proof of which is postponed to Section \ref{subsec:proof_Lemma_BVP}. Let $\mathcal D$ denote the open unit disc and let $\widetilde d$ be the discriminant  \eqref{eq:discriminants}.
\begin{lem}
\label{lem:BVP}
The function $D(y)$ can be analytically continued from the unit disc to the domain $\mathcal D\cup\mathcal{G}_\mathcal L$ and admits finite limits on $\mathcal{L}$. Moreover, $D(y)$ satisfies the following boundary condition, for $y\in\mathcal{L}$:
\begin{equation}
\label{eq:bvp_RCshift_3/4}
     \sqrt{\widetilde{d}(y)}D(y)-\sqrt{\widetilde{d}(\overline{y})}D(\overline{y})=y-\overline{y}.
\end{equation}
\end{lem}

In the remainder of the paper, we solve Lemma \ref{lem:BVP} in two different ways, leading to the contour-integral expressions of $D(y)$ given in Theorem~\ref{thm:first_main} and Theorem \ref{thm:second_main} below. Let us first remark that contrary to the usual quadrant case \cite{Ra-12}, the prefactor $\sqrt{\widetilde{d}(y)}$ in front of the unknown $D(y)$ is not meromorphic in $\mathcal{G}_\mathcal L$, simply because it is the square root of a polynomial, two roots of which being located in $\mathcal{G}_\mathcal L$ (see Section \ref{subsec:Roots_of_the_kernel}). This innocent-looking difference has strong consequences on the resolution:
\begin{itemize}
     \item Due to the presence of a non-meromorphic prefactor in \eqref{eq:bvp_RCshift_3/4}, solving the BVP of Lemma \ref{lem:BVP} requires the computation of an index (in the sense of Section \ref{sub:proof-thmbvpM1} and Appendix \ref{app:RHBVP}). This index is an integer and will be non-zero in our case, which will increase the complexity of the solutions. In Theorem~\ref{thm:first_main} we solve the BVP, by taking into account this non-zero index.
     \item A second, alternative idea is to reduce to the case of a meromorphic boundary condition, and thereby to an index equal to $0$. To do so, we will find an analytic function $f$ with the property that
\begin{equation}
\label{eq:decoupling_d_f}
     \frac{\sqrt{\widetilde{d}(\overline{y})}}{\sqrt{\widetilde{d}(y)}}=\frac{f(\overline{y})}{f(y)}
\end{equation}
for $y\in\mathcal L$, see Section \ref{sub:Anti-Tutte's} for more details. Such a function $f$ allows us to rewrite \eqref{eq:bvp_RCshift_3/4} as
\begin{equation}
\label{eq:other_way_BVP}
     f(y)D(y)-f(\overline{y})D(\overline{y})=\frac{f(y)}{\sqrt{\widetilde{d}(y)}}(y-\overline{y}),
\end{equation}
which by construction admits a meromorphic prefactor $f(y)$. In Theorem \ref{thm:second_main} we solve this zero-index BVP by this technique.
\end{itemize}
Although they represent the same function $D(y)$ (and so should be equal!), it will be apparent that the expressions obtained in Theorems \ref{thm:first_main} and \ref{thm:second_main} are quite different, and that the second one is simpler. However, we decided to present the two resolutions, as we think that they offer different insights on this boundary value method, and also because it is not obvious at all to be able to solve an equation of the form \eqref{eq:decoupling_d_f} and thereby to reduce to the zero-index case. Recall (Section \ref{subsec:Roots_of_the_kernel}) that $\mathcal{L}_0$ is the upper half of the curve $\mathcal{L}$.
\begin{thm}
\label{thm:first_main}
Let $w$ be a conformal gluing function with a pole at $y_2$. For any step set $\widehat{\mathcal{S}}$  satisfying to \ref{main_hyp}, the diagonal section \eqref{eq:def_diagonal} can be written, for $y\in \mathcal{G}_{\mathcal L}$,
\begin{equation*}
     D(y)=\frac{\Psi(w(y))}{2i\pi}\int_{\mathcal{L}_0}\frac{z-\overline{z}}{\sqrt{\widetilde{d}(z)}}\frac{w'(z)}{\Psi^+(w(z))(w(z)-w(y))}\dz,
\end{equation*} 
with
\begin{equation*}
\left\{
\begin{array}{r c l}
\Psi(y) &=& (y-Y(x_1))\exp{\Gamma(y)},\smallskip\\
\Psi^+(y) &=& (y-Y(x_1))\exp{\Gamma^+(y)},\smallskip\\
\Gamma(w(y))&=&	\displaystyle\frac{1}{2i\pi}\int_{\mathcal{L}_0}\log\left(\frac{\sqrt{\widetilde{d}(\overline{z})}}{\sqrt{\widetilde{d}(z)}}\right)\frac{w'(z)}{w(z)-w(y)}\dz.
\end{array}
\right.
\end{equation*}
All quantities are computed relative to the step set $\mathcal{S}=\varphi(\widehat{\mathcal{S}})$ after the change of coordinates \eqref{eq:change_var}.
\end{thm}
The left limit $\Gamma^+$ (and thereby $\Psi^+$) appearing in Theorem \ref{thm:first_main} can be computed with the help of Sokhotski-Plemelj formulas, that we have recalled in Proposition \ref{proposition:Sokhotski-Plemelj} of Appendix \ref{app:RHBVP}. We now turn to our second main result.
\begin{thm}
\label{thm:second_main}
Let $w$ be a conformal gluing function with a pole at $y_2$, with residue $r$. For any step set $\widehat{\mathcal{S}}$  satisfying to \ref{main_hyp}, the diagonal section \eqref{eq:def_diagonal} can be written, for $y\in \mathcal{G}_{\mathcal L}$,
\begin{equation*}
     D(y)= \frac{-w'(y)\sqrt{r}}{\sqrt{\widetilde d'(y_2)(w(y)-w(Y(x_1)))(w(y)-w(Y(x_2)))}}\frac{1}{2i\pi}\int_{\mathcal{L}}\frac{ zw'(z)}{\sqrt{w(z)-w(y_1)}(w(z)-w(y))}\dz.
\end{equation*}
All quantities are computed according to the step set $\mathcal{S}=\varphi(\widehat{\mathcal{S}})$.
\end{thm}

Here are some remarks about these results.

\smallskip

$\bullet$ First, it is important to notice that having an expression for $D(y)$ is sufficient for characterizing the complete generating function $C(x,y)$. Indeed, looking at Figure \ref{fig:some_sections} one is easily convinced that
\begin{equation*}
     C(x,y)=L(\varphi^{-1}(x,y))+D(\varphi^{-1}(x,y))+L(\varphi^{-1}(y,x)),
\end{equation*}
with
\begin{equation*}
\left\{
\begin{array}{r c l}
L(x,y) &=&\displaystyle \frac{1}{K(x,y)}\left(c(x)L_{-0}(x)-x\left(x\widetilde{a}(y)+\frac{1}{2}\widetilde{b}(y)\right)D(y)-\frac{1}{2}xy\right),\medskip\\
L_{-0}(x) &=&\displaystyle \frac{x}{c(x)}\left( \frac{1}{2}Y_0(x)+\left( x \widetilde{a}(Y_0(x))+\frac{1}{2}\widetilde{b}(Y_0(x))\right)D(Y_0(x)) \right),\medskip\\
\varphi^{-1}(x,y)&=&	(y^{-1},xy).
\end{array}
\right.
\end{equation*}

\smallskip

$\bullet$ Regarding the question of determining the algebraic nature of the diagonal series $D(y)$, the second expression is much simpler. Indeed, the integrand as well as the prefactor of the integral of Theorem \ref{thm:second_main} are algebraic functions of $y$, $z$, $t$ and $w$ (and its derivative) evaluated at various points. In addition, let us recall from \cite[Thm.~2]{Ra-12} that $w$ is algebraic if and only if the group is finite, and non-D-finite in the infinite group case. See Table \ref{tab:recap} for some implications. On the contrary, based on the exponential of a D-finite function, the integrand in Theorem \ref{thm:first_main} is a priori non-algebraic.

\begin{table}[h!]
\begin{center}
       \begin{tabular}{|l | l |l |l |}
\hline
Model &  Nature of $w$& Nature of $Q(x,y)$ & Nature of $C(x,y)$ \\
\hline \hline
$\smalldiagr{N,S,E,W}$ & rational \cite{Ra-12} & D-finite \cite{BMMi-10} & D-finite by \cite{BM-16} and Thm.\ \ref{thm:second_main} \\\hline
$\smalldiagr{S,W,NE}\ \smalldiagr{E,N,SW}\ \smalldiagr{E, NE, N, W, SW, S}$ & algebraic \cite{Ra-12}  & algebraic \cite{BMMi-10} & D-finite by Thm.\  \ref{thm:second_main}; algebraic?\\\hline
\multirow{2}{*}{$\smalldiagr{E, NE, N, SW}\ \smalldiagr{NE, W, SW, S}\ \smalldiagr{E, N, W, SW, S}\ \smalldiagr{E, NE, N, W, S}$} & \multirow{2}{*}{non-D-finite \cite{Ra-12}} & \multirow{2}{*}{non-D-finite \cite{KuRa-12,BoRaSa-14,DrHaRoSi-17} }& non-D-finite in $t$ \cite{Mu-19};\\
&&&
non-D-finite in $x,y$?\\\hline
\end{tabular}
\end{center}
\caption{Algebraic nature of the conformal mapping $w$, the quadrant generating function $Q(x,y)$ and the three-quarter plane counting function $C(x,y)$}
\label{tab:recap}
\end{table}

\smallskip

$\bullet$ Lemma \ref{lem:BVP} entails that the function $D(y)$ can be analytically continued to the domain $\mathcal D\cup\mathcal{G}_\mathcal L$. This is apparent on the first statement (using properties of contour integrals). This is a little bit less explicit on Theorem \ref{thm:second_main}, because of the prefactor.

\smallskip

$\bullet$ Theorem \ref{thm:first_main} (resp.\ Theorem \ref{thm:second_main}) will be proved in Section \ref{sub:proof-thmbvpM1} (resp.\ Sections \ref{sub:Anti-Tutte's} and \ref{subsec:proof_thm_3/4_index0}).

\subsection{Simplification and series expansion in the reverse Kreweras case}

In this part we apply Theorem \ref{thm:second_main} to reverse Kreweras walks in the three-quarter plane: we first make explicit all quantities appearing in the statement of Theorem \ref{thm:second_main}, then we explain how to deduce the series expansion
\begin{equation}
\label{eq:series_expansion_reverse_Kreweras}
     D(0)=1+4\,{t}^{3}+46\,{t}^{6}+706\,{t}^{9}+12472\,{t}^{12}+239632\,{t}^{15}+4869440\,{t}^{18}+102995616\,{t}^{21}+O \left( {t}^{24} \right),
\end{equation}
obtained here by direct enumeration. Let us recall that the coefficients in front of $t^{n}$ are the $c_{0,0}(n)$, which count the numbers of reverse Kreweras walks of length $n$, starting and ending at $(0,0)$ and confined to the three-quarter plane.

 This symmetric model has the step set $\widehat{\mathcal S}=\{ (1,0), (0,1), (-1,-1) \}$, see Figure \ref{fig:symmetric_models_finite}. The change of variable $\varphi$ defined in \eqref{eq:change_var} transforms it into Kreweras step set, see Figure \ref{fig:symmetric_models_finite} and Table \ref{tab:transform_models}, with $\mathcal S=\{ (1,1), (-1,0), (0,-1)\}$.

\paragraph{Computation of various quantities.}

 The kernel \eqref{eq:kernel} is
$
     K(x,y)=xy(t(xy+x^{-1}+y^{-1}) -1), 
$
and with the notations \eqref{eq:coeff_kernel} and \eqref{eq:discriminants}, we have
\begin{equation*}
     a(x)=tx^{2},\quad b(x)=t-x,\quad c(x)=tx, \quad d(x)=(t-x)^{2}-4t^{2}x^{3},
\end{equation*}
and by symmetry $\widetilde a=a$, $\widetilde b=b$, $\widetilde c=c$ and $\widetilde d=d$.
The branch points $x_{1}$ and $x_{2}$ are the roots of $d$ in the open unit disc, such that $x_{1}<x_{2}$. We easily obtain
\begin{equation*}
\left\{
\begin{array}{r c l}
x_{1}\ =\ y_1&\hspace{-1.5mm}=\hspace{-1.5mm}&\displaystyle t-2\,{t}^{5/2}+6\,{t}^{4}-21\,{t}^{11/2}+80\,{t}^{7}-{\frac {1287}{4}}{t}^{17/2}+O ( {t}^{10} ),\medskip\\
x_{2}\ =\ y_2&\hspace{-1.5mm}=\hspace{-1.5mm}&\displaystyle t+2\,{t}^{5/2}+6\,{t}^{4}+21\,{t}^{11/2}+80\,{t}^{7}+{\frac {1287}{4}}{t}^{17/2}+O ( {t}^{10} ).
\end{array}\right.
\end{equation*}
We further have
\begin{equation*}
     \widetilde{d}'(y_2) =  2\,{t}^{5/4}-\frac{3}{2}\,{t}^{{17/4}}-8\,{t}^{{23/4}}-{\frac {603}{16}{t}^{{29/4}}}-174\,{t}^{{35/4}}+O \left( {t}^{{41/4}} \right).
\end{equation*}
We finally need to compute $Y(x_{1})$ and $Y(x_{2})$. By \eqref{eq:algebraic_expressions_Y_X} these quantities may be simplified as
\begin{equation*}
Y(x_{1})=-\sqrt{\frac{c(x_{1})}{a(x_{1})}}=-\sqrt{\frac{1}{x_{1}}} \qquad \text{and}\qquad Y(x_{2})=\sqrt{\frac{c(x_{2})}{a(x_{2})}}=\sqrt{\frac{1}{x_{2}}}.
\end{equation*}

\paragraph{Expression of the conformal gluing function.}

As we shall prove in Lemma \ref{lem:list_conformal_gluing_functions}, the following is a suitable conformal mapping:
\begin{equation*}
     w(y) = \left(\frac{1}{y}-\frac{1}{W}\right)\sqrt{1-yW^2},
\end{equation*}
where $W$ is the unique power series solution to $W=t(2+W^3)$. As Theorem \ref{thm:second_main} is stated for a conformal gluing function with a pole at $y_{2}$ and not at $0$, we should consider instead $w_{y_{2}}=\frac{1}{w-w(y_{2})}$.
We will need the following expansions:
\begin{equation*}
\left\{
\begin{array}{r c l}
     W&\hspace{-1.5mm}=\hspace{-1.5mm}&\displaystyle 2\,t+8\,{t}^{4}+96\,{t}^{7}+1536\,{t}^{10}+O \left( {t}^{11} \right),\medskip\\
     w_{y_{2}}(y_1)&\hspace{-1.5mm}=\hspace{-1.5mm}&\displaystyle \,{\frac {1}{4}}t^{-1/2}-\frac{3}{8}\,{t}^{5/2}-{\frac {97}{32}}{t}^{11/2}-{\frac {2611}{64}}{t}^{17/2}+O \left( {t}^{23/2} \right),\medskip\\
     w(y_2)&\hspace{-1.5mm}=\hspace{-1.5mm}&\displaystyle  \frac{1}{2}t^{-1}-2\, {t}^{1/2}-{t}^{2}-3\,{t}^{7/2}-7\,{t}^{5}-{\frac {115}{4}}{t}^{13/2}-90\,{t}^{8}-{\frac {3247}{8}}{t}^{19/2}+O \left( {t}^{11} \right),\medskip\\
     w_{y_{2}}(Y(x_1))&\hspace{-1.5mm}=\hspace{-1.5mm}&\displaystyle -t-2\,{t}^{4}-18\,{t}^{7}+O \left( {t}^{10} \right),\medskip\\
     w_{y_{2}}(Y(x_2))&\hspace{-1.5mm}=\hspace{-1.5mm}&\displaystyle -t-4\,{t}^{5/2}-18\,{t}^{4}-86\,t^{11/2}-418\,{t}^{7}-{\frac {4131}{2}}{t}^{17/2}+O \left( {t}^{10} \right),\medskip\\
     w'(y_2)&\hspace{-1.5mm}=\hspace{-1.5mm}&\displaystyle {t}^{-1}-2\,{t}^{1/2}-5/2\,{t}^{2}-6\,{t}^{7/2}-{\frac {169}{8}}{t}^{5}-75\,{t}^{13/2}-{\frac {4957}{16}}{t}^{8}-1251\,{t}^{19/2}+O \left( {t}^{11} \right). 
\end{array}\right.
\end{equation*}

 \paragraph{Explicit expression of $D(y)$.}

We apply now Theorem \ref{thm:second_main} and obtain
\begin{multline*}
     D(y)= \frac{-w_{y_{2}}'(y)}{\sqrt{(w_{y_{2}}(y)-w_{y_{2}}(-1/\sqrt{x_1}))(w_{y_{2}}(y)-w_{y_{2}}(1/\sqrt{x_2}))\widetilde{d}'(y_2)w'(y_2)}}
 \\
     \frac{1}{2i\pi}\int_{\mathcal{L}}\frac{ zw_{y_{2}}'(z)}{\sqrt{w_{y_{2}}(z)-w_{y_{2}}(y_1)}(w_{y_{2}}(z)-w_{y_{2}}(y))}\dz,
\end{multline*}
where $\mathcal{L}$ is the contour defined in \eqref{eq:curve_L}, represented on Figure \ref{fig:curve_Kreweras}. Since $w_{y_{2}}(0)=0$ and $w_{y_{2}}'(0)=-1$ (remember that $w$ has a pole at $y=0$), evaluating at $y=0$ the expression above yields
\begin{equation}
\label{eq:expression_D(0)}
     D(0)= \frac{1}{\sqrt{w_{y_{2}}(-1/\sqrt{x_1})w_{y_{2}}(1/\sqrt{x_2})\widetilde{d}'(y_2)w'(y_2)}}
     \frac{1}{2i\pi}\int_{\mathcal L}\frac{ zw_{y_{2}}'(z)}{\sqrt{w_{y_{2}}(z)-w_{y_{2}}(y_1)}w_{y_{2}}(z)}\dz.
\end{equation}
The integrand in the right-hand side of the above equation is analytic on $\mathcal G_\mathcal L\setminus [y_1,y_2]$. Hence by Cauchy's integral theorem, the contour $\mathcal L$ may be replaced by the unit circle $\mathcal C(0,1)$.

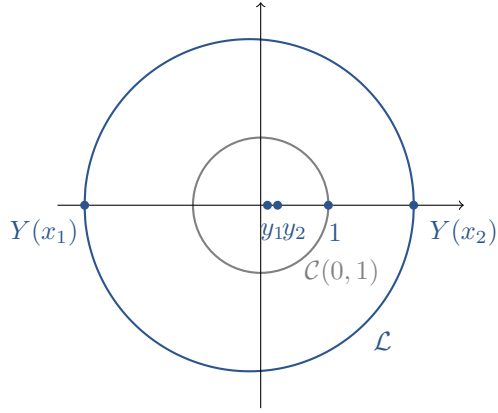
\begin{figure}[t]
  \centering
  \begin{tikzpicture}
\begin{scope}[scale=0.9]
\draw [dblue!90, thick] plot[smooth] file{./figures/KrewY0.dat};
\draw [dblue!90, thick] plot[smooth] file{./figures/KrewY1.dat};
\draw [gray, thick] (0,0) circle (1cm);
\draw [->] (-3,0) -- (3,0);
\draw [->] (0,-3) -- (0,3);
\draw[thick, dblue!90, fill=dblue!90] (1,0) circle (0.05 cm);
\draw [dblue!90] (1.1,-0.4) node{\small{$1$}};
\draw[thick, dblue!90, fill=dblue!90] (0.1,0) circle (0.05 cm);
\draw [dblue!90] (0.18,-0.4) node{\small{$y_{1}$}};
\draw[thick, dblue!90, fill=dblue!90] (0.25,0) circle (0.05 cm);
\draw [dblue!90] (0.5,-0.4) node{\small{$y_{2}$}};
\draw[thick, dblue!90, fill=dblue!90] (-2.6,0) circle (0.05 cm);
\draw [dblue!90] (-3.2,-0.4) node{\small{$Y(x_{1})$}};
\draw[thick, dblue!90, fill=dblue!90] (2.26,0) circle (0.05 cm);
\draw [dblue!90] (3,-0.4) node{\small{$Y(x_{2})$}};
\draw [gray] (1.2, -1) node{\small{$\mathcal{C}(0,1)$}};
\draw [dblue!90] (1.8, -2) node{\small{$\mathcal{L}$}};
\end{scope}
\end{tikzpicture}
   \caption{The curve $\mathcal{L}$ for Kreweras model, for $t=1/6$}
  \label{fig:curve_Kreweras}
\end{figure}

 \paragraph{Expression of $D(0)$ as a function of $W$.}
We could directly make a series expansion of $D(0)$ in $t$. However, for greater efficiency of the series expansion computation, we will first express $D(0)$ in terms of $W$, expand this integral in a series of $W$ and finally get back to a series in $t$. 
The generating function of excursions $D(0)$ can be written as 
\begin{equation}
\label{eq:D0-fct-W}
D(0)=-\frac{\sqrt{w(y_{1})-w(y_{2})}}{\sqrt{w_{y_{2}}(-1/\sqrt{x_{1}})w_{y_{2}}(1/\sqrt{x_{2}})\widetilde{d}'(y_{2})w'(y_{2})}}\frac{1}{2i\pi}\int_{\mathcal L}\frac{zw'(z)}{\sqrt{P-Sw(z)+w(z)^{2}}}\dz,
\end{equation}
with 
\begin{equation}
\label{eq:P-S}
\left\{
\begin{array}{l c l c l }
S &=& w(y_{1})+w(y_{2}) &=& \displaystyle\sqrt{2P-\frac{1}{4W^{2}}\left(W^{6}-20W^{3}-8 \right)},
\\
P &=& w(y_{1})w(y_{2}) &=& \displaystyle\frac{(1-W^{3})^{3/2}}{W^{2}}.
\end{array}
\right.
\end{equation}
In order to derive \eqref{eq:D0-fct-W}, we start by writing the integrand of \eqref{eq:expression_D(0)} in terms of $w$:
\begin{equation*}
\int_{\mathcal L}\frac{zw'_{y_{2}}(z)}{\sqrt{w_{y_{2}}(z)-w_{y_{2}}(y_{1})}w_{y_{2}}(z)}\dz=-\sqrt{w(y_{1})-w(y_{2})}\int_{\mathcal L}\frac{zw'(z)}{\sqrt{(w(z)-w(y_{1}))(w(z)-w(y_{2}))}}\dz.
\end{equation*}
Then, note that $\widetilde{d}(y)=-4t^{2}(y-y_{1})(y-y_{2})(y-y_{3})=-4t^{2}(y-y_{1})(y-y_{2})\left(y-\frac{1}{W^{2}}\right)$. By identification we have $y_{1}+y_{2}=\frac{1}{4t^{2}}-\frac{1}{W^{2}}$ and $y_{1}y_{2}=\frac{W^{2}}{4}$. On the one hand, we can deduce that 
\begin{equation*}
P=\left(\frac{1}{y_{1}}-\frac{1}{W}\right)\left(\frac{1}{y_{2}}-\frac{1}{W}\right)\sqrt{(1-y_{1}W^{2})(1-y_{2}W^{2})}=\frac{-(W-2t)(W-3t)}{W^{5}t^{2}}\sqrt{W^{6}t^{2}-W^{2}+8t^{2}}.
\end{equation*}
On the other hand, 
\begin{equation*}
S^{2}=\left(1-y_{1}W^{2}\right)\left(\frac{1}{y_{1}^{2}}-\frac{2}{y_{1}W}+\frac{1}{W^{2}}\right)+\left(1-y_{2}W^{2}\right)\left(\frac{1}{y_{2}^{2}}-\frac{2}{y_{2}W}+\frac{1}{W^{2}}\right)+2P.
\end{equation*}
Both equations can be simplified into \eqref{eq:P-S}, using several times the minimal polynomial of $W$.

\paragraph{Series expansion.}

Let us first expand in $t$ the factor in front of the integral in \eqref{eq:D0-fct-W}; we get
\begin{equation*}
-\frac{\sqrt{w(y_{1})-w(y_{2})}}{\sqrt{w_{y_{2}}(-1/\sqrt{x_1})w_{y_{2}}(1/\sqrt{x_2})\widetilde{d}'(y_2)w'(y_2)}}=-\frac{1}{t}+O \left( {t}^{{10}} \right).
\end{equation*}
(One could even prove that the left-hand side of the above equation is identically equal to $-\frac{1}{t}$.) Then the factor in the integral in \eqref{eq:D0-fct-W} may be written as
\begin{align*}
\frac{zw'(z)}{\sqrt{P-Sw(z)+w(z)^{2}}}=&-\frac{1}{2z}\,{W}+ \left( -\frac{1}{4z^{2}}\,+\frac{z}{4} \right) {W}^{2}-\frac{1}{8{z}^{3}}\,{W}^{3}+ \left( -\frac{1}{16{z}^{4}}\,+\frac{3z^{2}}{16}\,\right) {W}^{4}
 \\
 &+ \left( -\frac{z}{32}-\frac{1}{16z^{2}}\,-\frac{1}{32z^{5}}\, \right) {
W}^{5}+ \left( -{\frac {3}{32\,{z}^{3}}}-{\frac {1}{64\,{z}^{6}}}+{
\frac {5\,{z}^{3}}{32}} \right) {W}^{6}
\\
&+ \left( -\frac{z^{2}}{32}\,+{\frac 
{1}{64\,z}}-{\frac {3}{32\,{z}^{4}}}-{\frac {1}{128\,{z}^{7}}}
 \right) {W}^{7}
 \\
 &+ \left( {\frac {1}{64\,{z}^{2}}}-{\frac {1}{256\,{z}^
{8}}}-{\frac {5}{64\,{z}^{5}}}-{\frac {z}{128}}+{\frac {35\,{z}^{4}}{
256}} \right) {W}^{8}
\\
&+ \left( -{\frac {1}{512\,{z}^{9}}}-{\frac {15}{
256\,{z}^{6}}}-{\frac {15\,{z}^{3}}{512}} \right) {W}^{9}
+O \left( {W}^{{10}} \right)
\end{align*}
and when we integrate the latter on the unit circle. Coming back to a series in $t$ we obtain
\begin{multline*}
 \frac{1}{2i\pi}\int_{\mathcal C(0,1)}\frac{zw'(z)}{\sqrt{P-Sw(z)+w(z)^{2}}}\dz=\\-t-4\,{t}^{4}-46\,{t}^{7}-706\,{t}^{10}-12472\,{t}^{13}
  -239632\,{t}^{
16}-4869440\,{t}^{19}+O \left( {t}^{21} \right).
\end{multline*}
Finally, putting every ingredients in order, we deduce \eqref{eq:series_expansion_reverse_Kreweras}.

\subsection{Proof of Lemma \ref{lem:BVP}}
\label{subsec:proof_Lemma_BVP}

Assuming that $D(y)$ may be continued as in the statement of Lemma \ref{lem:BVP}, it is easy to prove the boundary condition \eqref{eq:bvp_RCshift_3/4}. We evaluate the functional equation \eqref{eq:functional_equation_octant} at $Y_0(x)$ for $x$ close to $[x_1,x_2]$:
\begin{equation}
\label{cancel_kernel_3/4}
     -\frac{1}{2}xY_0(x)+c(x)L_{-0}(x)-x(x\widetilde{a}( Y_0(x))+\frac{1}{2}\widetilde{b}(Y_0(x)) )D(Y_0(x))+\frac{1}{2}txD(0)=0.
\end{equation}
We obtain two new equations by letting $x$ go to any point of $[x_1,x_2]$ with a positive (resp.\ negative) imaginary part. We do the subtraction of the two equations and obtain \eqref{eq:bvp_RCshift_3/4}. 

We now prove the analytic continuation. Note that similar results are obtained in \cite[Thm.~3.2.3]{FaIaMa-17}, \cite[Thm.~5]{Ra-12} and \cite[Prop.~19]{BeBMRa-17}. We follow the same idea as in \cite[Thm.~5]{Ra-12}. Starting from \eqref{eq:functional_equation_octant} we can prove that
\begin{equation*}
     2c(X_0(y))L_{-0}(X_0(y))+X_0(y)\sqrt{\widetilde{d}(y)}D(y)-X_0(y)y=0
\end{equation*}
for $y\in \{y\in \mathbb{C} : \vert X_0(y)\vert<1\} \cap \mathcal{D}$, and then  
\begin{equation*}
     2c(X_0(y))\sum_{n\geq 0,j\geq 0}c_{0,-j-1}(n)X_0(y)^j t^n+\sqrt{\widetilde{d}(y)}D(y)-y=0
\end{equation*}
for $y\in \{y\in \mathbb{C} : |X_0(y)|<1 \text{ and }  X_0(y)\neq 0\} \cap \mathcal{D}$ which can be continued in $ \mathcal{G}_{\mathcal{L}} \cup \mathcal{D}$. Being a power series, $D(y)$ is analytic on $\mathcal{D}$ and on $(\mathcal{G}_{\mathcal{L}} \cup \mathcal{D})\setminus \mathcal{D}$, $D(y)$ may have the same singularities as $X_0$ and $\sqrt{\widetilde{d}(y)}$, namely the branch cuts $[y_1,y_2]$ and $[y_3,y_4]$. But none of these segments belong to $(\mathcal{G}_{\mathcal{L}} \cup \mathcal{D})\setminus \mathcal{D}$, see Lemma \ref{lem:properties_curves}. Then $D(y)$ can be analytically continued to the domain $\mathcal{G}_{\mathcal{L}} \cup \mathcal{D}$. Using the same idea, we can prove that $D(y)$ has finite limits on $\mathcal{L}$. From  \eqref{cancel_kernel_3/4},  it is enough to study the zeros of $x\widetilde{a}( Y_0(x))+\frac{1}{2}\widetilde{b}(Y_0(x))$ for $x$ in $[x_1,x_2]$. Using the relation $X_0(Y_0(x))=x$ valid in $\mathcal G_\mathcal M$ (see \cite[Cor.~5.3.5]{FaIaMa-17}) shows that it recurs to study the zeros of $\widetilde{d}(y)$ for $y\in(\mathcal{G}_{\mathcal{L}} \cup \mathcal{D})\setminus \mathcal{D}$. None  of these roots ($y_1, y_2, y_3, y_4$) belong to the last set, then $D$ has finite limits on $\mathcal{L}$.

\subsection{Proof of Theorem \ref{thm:first_main}}
\label{sub:proof-thmbvpM1}

The function $\sqrt{\widetilde{d}(y)}D(y)$ satisfies a BVP of Riemann-Carleman type on $\mathcal{L}$, see Lemma \ref{lem:BVP}. Following the literature \cite{FaIaMa-17,Ra-12}, we use a conformal mapping to transform the latter into a more classical Riemann-Hilbert BVP. Throughout this section, we shall use notation and results of Appendix~\ref{app:RHBVP}.

More precisely, let $w$ be a conformal gluing function for the set $\mathcal{G}_{\mathcal{L}}$ in the sense of Definition \ref{defn:CGF}, and let $\mathcal{U}$ denote the real segment 
\begin{equation*}
     \mathcal{U} = w(\mathcal{L}).
\end{equation*}
(With this notation, $w$ is a conformal mapping from $\mathcal{G}_{\mathcal{L}}$ onto the cut plane $\mathbb{C} \setminus \mathcal{U}$.) The segment $\mathcal{U}$ is oriented such that the positive direction is from $w(Y(x_2))$ to $w(Y(x_1))$, see Figure \ref{fig:conformal-functions}. 

Define $v$ as the inverse function of $w$. The latter is meromorphic on $\mathbb{C} \setminus \mathcal{U}$. Following the notation of Appendix~\ref{app:RHBVP} and \cite{FaIaMa-17}, we denote by $v^+$ and $v^-$ the left and right limits of $v$  on $\mathcal{U}$. The quantities $v^+$ and $v^-$ are complex conjugate on $\mathcal{U}$, and more precisely, since $w$ preserves angles, we have for $u\in\mathcal{U}$ and $y\in \mathcal{L}_0$ 
\begin{equation*}
\left\{
\begin{array}{l c l c l}
v^+(u) &=& v^+(w(y)) &=& y, \\
v^-(u) &=& v^-(w(y)) &=& \overline{y},
\end{array}
\right.
\end{equation*}
see Figure \ref{fig:conformal-functions} for an illustration of the above properties. 

Then \eqref{eq:bvp_RCshift_3/4} may be rephrased as the following new boundary condition on $\mathcal{U}$:
\begin{equation}
\label{eq:pbl_RH}
     D(v^+(u))=\frac{\sqrt{\widetilde{d}( v^-(u))}}{\sqrt{\widetilde{d}( v^+(u))}} D(v^-(u)) + \frac{v^+(u) - v^-(u)}{\sqrt{\widetilde{d}( v^+(u))}}.
\end{equation}
As explained in Appendix \ref{app:RHBVP} (see in particular Definition \ref{def:index}), the first step in the way of solving the Riemann-Hilbert problem with boundary condition \eqref{eq:pbl_RH} is to compute the index of the BVP.

\begin{proposition}
\label{prop:computation_index}
The index of $\frac{\sqrt{\widetilde{d}( v^-(u))}}{\sqrt{\widetilde{d}( v^+(u))}}$ along the curve $\mathcal U$ is $-1$.
\end{proposition}

\begin{proof}
First of all, let us recall that when $\mathcal{L}$ is a closed curve of interior $\mathcal{G}_{\mathcal{L}}$ and $G$ is a non-constant, meromorphic function without zeros or poles on $\mathcal{L}$, then
\begin{equation*}
\Ind_\mathcal{L} G=\frac{1}{2i\pi}\int_\mathcal{L}\frac{G'(z)}{G(z)}\dz=Z-P,
\end{equation*} 
where $Z$ and $P$ are respectively the numbers of zeros and poles of $G$ in $\mathcal{G}_{\mathcal{L}}$, counted with multiplicity. 

Applying this result to the function $d(y)$, which in $\mathcal{G}_{\mathcal L}$ has no pole and exactly two zeros (at $y_1$ and $y_2$---remember that $y_3$ and $y_4$ are also roots of $d(y)$ but are not in $\mathcal{G}_{\mathcal L}$), we have  $\Ind_{\mathcal{L}}\widetilde{d}(y)=2$, see Figure \ref{fig:index_Gessel} for an illustration.

We get then
\begin{align*}
\Ind_{\mathcal{U}}\frac{\sqrt{\widetilde{d}(v^-(u))}}{\sqrt{\widetilde{d}(v^+(u))}}&=
\Ind_{\mathcal{U}}\sqrt{\widetilde{d}(v^-(u))} - \Ind_{\mathcal{U}}\sqrt{\widetilde{d}(v^+(u))}
=-\Ind_{\mathcal{L}_1}\sqrt{\widetilde{d}(y)}-\Ind_{\mathcal{L}_0}\sqrt{\widetilde{d}(y)}
\\
&=-\Ind_{\mathcal{L}}\sqrt{\widetilde{d}(y)}
=-\frac{1}{2}\Ind_{\mathcal{L}}\widetilde{d}(y)
=-1.\qedhere
\end{align*}
\end{proof}

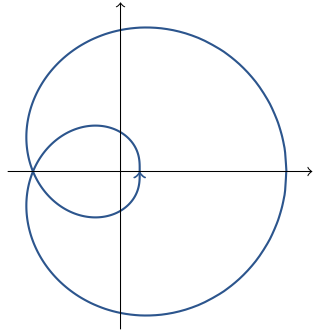
\begin{figure}[htb]
  \centering
  \begin{tikzpicture}
\begin{scope}[scale=0.015]
\draw [dblue!90, thick, ->] plot[smooth] file{./figures/R0.dat};
\draw [dblue!90, thick] plot[smooth] file{./figures/R1.dat};
\draw [->] (0,-140) -- (0,150);
\draw [->] (-100,0) -- (170,0);
\end{scope}
\end{tikzpicture}
   \caption{Plot of $\widetilde{d}(y)$ when $y$ lies on $\mathcal{L}$, in the case of Gessel's step set}
  \label{fig:index_Gessel}
\end{figure}

With Theorem \ref{thm:Sol-BVP}, we deduce a contour-integral expression for the function $D(v(u))$, namely
\begin{equation*}
     D(v(u))=\frac{\Psi(u)}{2i\pi}\int_{\mathcal{U}}\frac{v^+(s) - v^-(s)}{\sqrt{\widetilde{d}\left( v^+(s)\right)}}\frac{1}{\Psi^+(s)(s-u)}\ds.
\end{equation*}
With the changes of variable $u=w(y)$ and $s=w(z)$, we easily have the result of Theorem \ref{thm:first_main}.

\subsection{Anti-Tutte's invariant}
\label{sub:Anti-Tutte's}

Our aim here is to find a function $f$ satisfying to the decoupling condition \eqref{eq:decoupling_d_f}, namely
\begin{equation*}
     \frac{\sqrt{\widetilde{d}(\overline{y})}}{\sqrt{\widetilde{d}(y)}}=\frac{f(\overline{y})}{f(y)},\qquad \forall y\in \mathcal L.
\end{equation*}
Indeed, such a function is used in a crucial way in Theorem \ref{thm:second_main}.

Before giving a systematic construction of a function $f$ as above, we start by an example. For Gessel's model, we easily prove that the function
\begin{equation*}
     g(y) = \frac{y}{t(y+1)^2}
\end{equation*}
satisfies $g(Y_0)g(Y_1)=1$, and so for $x\in[x_1,x_2]$ the condition \eqref{eq:anti_Tutte} announced in the introduction. By the same reasoning as in the proof of Theorem \ref{thm:solving_decoupling_condition} below, we deduce that
\begin{equation*}
     f(y)=\frac{g(y)}{g'(y)}=\frac{y(y+1)}{y-1}
\end{equation*}
satisfies the decoupling condition \eqref{eq:decoupling_d_f}. 

However, a simple rational expression of $f$ as above does not exist in general. Instead, our general construction consists in writing $f$ in terms of a conformal mapping. Our main result is the following. 
\begin{thm}
\label{thm:solving_decoupling_condition}
Let $g$ be any conformal mapping from $\mathcal{G}_\mathcal L$ onto the unit disc $\mathcal D$, with the property that $g(\overline{y})=\overline{g(y)}$. Then the function $f$ defined by 
\begin{equation*}
     f=\frac{g}{g'}
\end{equation*}
satisfies the decoupling condition \eqref{eq:decoupling_d_f}. Moreover, $f$ is analytic in $\mathcal{G}_\mathcal L$ and has finite limits on $\mathcal L$. 

Finally, defining $h(z)=-z+\sqrt{z^2-1}$ and letting $w$ be a conformal gluing function as in Definition~\ref{defn:CGF}, one can choose
\begin{equation}
\label{eq:expression_g}
     g(y)=h\left(\frac{2}{w(Y(x_2))-w(Y(x_1))}\left(w(y)-\frac{w(Y(x_1))+w(Y(x_2))}{2}\right)\right),
\end{equation}
see Figure \ref{fig:conformal-functions}.
\end{thm}

To obtain the expression of $g$ in \eqref{eq:expression_g} for a given model, we refer to the list of conformal mappings $w$ provided in Appendix \ref{app:expression_Gluing}.

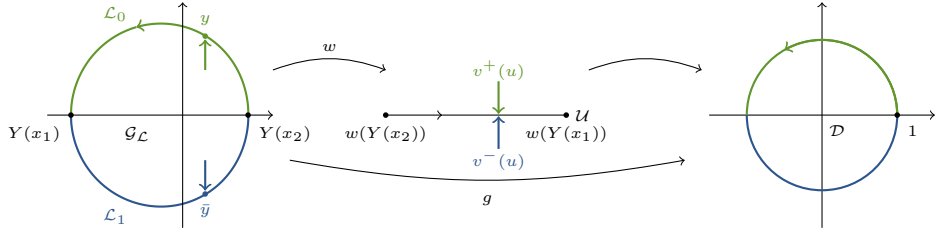
\begin{figure}[htb]
\centering
\begin{tikzpicture}
\begin{scope}[scale=0.15]
\draw [dblue!90, thick] plot[smooth] file{./figures/Y1.dat};
\draw [<-, dgreen!90, thick] plot[smooth] file{./figures/Y0bis1.dat};
\draw [dgreen!90, thick] plot[smooth] file{./figures/Y0bis2.dat};
\draw [->] (0,-10) -- (0,10);
\draw [->] (-12,0) -- (8,0);
\draw[thick, black, fill=black](-9.9,0) circle (0.15 cm)  node[below left]{\tiny{$Y(x_1)$}};
\draw[thick, black, fill=black](5.8,0) circle (0.15 cm)  node[below right]{\tiny{$Y(x_2)$}};
\draw [black] (-4, -1.7) node{\tiny{$\mathcal G_\mathcal L$}};
\draw [dgreen!90] (-6,9) node{\tiny{$\mathcal{L}_0$}};
\draw [dblue!90] (-6,-9) node{\tiny{$\mathcal{L}_1$}};
\draw[thick, dgreen!90, fill=dgreen!90] (2,7) circle (0.15 cm) node[above]{\tiny{$y$}};
\draw[thick, dblue!90, fill=dblue!90] (2,-7) circle (0.15 cm) node[below]{\tiny{$\bar{y}$}};
\draw [->, thick, dgreen!90] (2, 4) -- (2, 6.7);
\draw [->, thick, dblue!90] (2, -4) -- (2, -6.7);
\end{scope}

\begin{scope}[scale=0.15, xshift=6cm]
\draw[->] (2,4) to[out=20,in=160] (12,4);
\draw [] (7,6) node{\tiny{$w$}};
\draw[->] (30,4) to[out=20,in=160] (40,4);
\draw[->] (3.5,-4) to[out=350,in=190] (38.5,-4) ;
\draw [] (21,-7.5) node{\tiny{$g$}};
\end{scope}

\begin{scope}[scale=0.15, xshift=18cm]
\draw [->] (0,0) -- (5,0);
\draw [] (5,0) -- (16,0);
\draw [] (17.5,0) node{\tiny{$\mathcal{U}$}};
\draw [dgreen!90] (10,4) node{\tiny{$v^+(u)$}};
\draw [dblue!90] (10,-4) node{\tiny{$v^-(u)$}};
\draw [->, thick, dgreen!90] (10, 3) -- (10, 0.1);
\draw [->, thick, dblue!90] (10, -3) -- (10, -0.1);
\draw[thick, black, fill=black](0,0) circle (0.15 cm)  node[below]{\tiny{$w(Y(x_2))$}};
\draw[thick, black, fill=black](16,0) circle (0.15 cm)  node[below]{\tiny{$w(Y(x_1))$}};
\end{scope}

\begin{scope}[scale=1,  xshift=8.5cm]
\draw [->] (-1.5,0) -- (1.5,0);
\draw [->] (0,-1.5) -- (0,1.5);
\draw[thick, dgreen!90, shift={(1,0)}]  (0:0) arc (0:180:1);
\draw[->, thick, dgreen!90, shift={(1,0)}]  (0:0) arc (0:120:1);
\draw[thick, dblue!90, shift={(1,0)}]  (0:0) arc (0:-180:1);
\draw[thick, black, fill=black](1,0) circle (0.0225 cm)  node[below right]{\tiny{$1$}};
\draw [black] (0.2,-0.2) node{\tiny{$\mathcal{D}$}};
\end{scope}
\end{tikzpicture}
\caption{Conformal gluing functions from $\mathcal G_\mathcal L$ to $\mathbb C\setminus \mathcal U$ and conformal mappings from $\mathcal G_\mathcal L$ to the unit disc~$\mathcal D$}
\label{fig:conformal-functions}
\end{figure}

\begin{proof}
We first prove that if $g$ is a conformal mapping from $\mathcal{G}_\mathcal L$ onto the unit disc $\mathcal D$ with the property that $g(\overline{y})=\overline{g(y)}$, then $f=\frac{g}{g'}$ satisfies the decoupling condition \eqref{eq:decoupling_d_f}. First, for $x\in[x_1,x_2]$ one has 
\begin{equation*}
     g(Y_0(x))g(Y_1(x))=g(Y_0(x))g(\overline{Y_0(x)})=g(Y_0(x))\overline{g(Y_0(x))}=\vert g(Y_0(x))\vert^2=1.
\end{equation*}
Differentiating the identity $g(Y_0(x))g(Y_1(x))=1$, one finds on $[x_1,x_2]$
\begin{equation*}
     \frac{f(Y_0(x))}{f(Y_1(x))}= - \frac{Y_0'(x)}{Y_1'(x)}.
\end{equation*}
To conclude the proof, we show that on $[x_1,x_2]$
\begin{equation}
\label{eq:decoupling_sqrt_d}
     \frac{\sqrt{\widetilde d(Y_0(x))}}{\sqrt{\widetilde d(Y_1(x))}}= - \frac{Y_0'(x)}{Y_1'(x)}.
\end{equation}
To that purpose, let us first consider $x\in\mathcal G_\mathcal M\setminus [x_1,x_2]$. Differentiating the identity $K(x,Y_0(x))=0$ in \eqref{eq:kernel_expanded} yields
\begin{equation}
\label{eq:first_id}
     Y'_0(x)(2a(x)Y_0(x)+b(x))=-(a'(x)Y_0(x)^2+b'(x)Y_0(x)+c'(x)).
\end{equation}
First, it follows from Section~\ref{subsec:Roots_of_the_kernel} that $2a(x)Y_0(x)+b(x)=- \sqrt{d(x)}$. Moreover, differentiating \eqref{eq:kernel_expanded} in $x$ and using the relation $X_0(Y_0(x))=x$ valid in $\mathcal G_\mathcal M$ (see \cite[Cor.~5.3.5]{FaIaMa-17}) shows that the right-hand side of \eqref{eq:first_id} satisfies
\begin{equation*}
     a'(x)Y_0(x)^2+b'(x)Y_0(x)+c'(x)=-\sqrt{\widetilde{d}(Y_0(x))}.
\end{equation*}
Then for $x\in\mathcal G_\mathcal M\setminus [x_1,x_2]$, Equation \eqref{eq:first_id} becomes
\begin{equation*}
     -\sqrt{d(x)}Y_0'(x)=\sqrt{\widetilde{d}(Y_0(x))}.
\end{equation*}
To complete the proof of \eqref{eq:decoupling_sqrt_d}, we let $x$ converge to a point $x\in[x_1,x_2]$ from above and then from below, and we compute the ratio of the two identities so-obtained. The minus sign in \eqref{eq:decoupling_sqrt_d} comes from that 
\begin{equation*}
     \lim_{x\downarrow[x_1,x_2]} \sqrt{d(x)}=-\lim_{x\uparrow[x_1,x_2]} \sqrt{d(x)},
\end{equation*}
see Section~\ref{subsec:Roots_of_the_kernel}.

\medskip

Our second point is to show that the function $g$ in \eqref{eq:expression_g} is a conformal mapping from $\mathcal{G}_\mathcal L$ onto the unit disc $\mathcal D$, which in addition is such that $g(\overline{y})=\overline{g(y)}$. This is obvious from our construction \eqref{eq:expression_g}, since as illustrated on Figure \ref{fig:conformal-functions}, $g=h\circ \widehat w$ is the composition of the conformal mapping $h$ from the cut plane $\mathbb C\setminus [-1,1]$ onto the unit disc, by the conformal mapping 
\begin{equation}
\label{eq:def_hat_w}
     \widehat w=\frac{2}{w(Y(x_1))-w(Y(x_2))}\left(w-\frac{w(Y(x_1))+w(Y(x_2))}{2}\right)
\end{equation}
from $\mathcal G_\mathcal L$ onto the same cut plane.

\medskip

The third item is to prove that $f$ has finite limits on $\mathcal L$, for any initial choice of conformal mapping $g$. We may propose two different proofs of this fact. First, we could prove that the function $f$ constructed from the particular function $g$ in \eqref{eq:expression_g} has the desired properties (this follows from a direct study). Then as any two suitable conformal mappings $g_1$ and $g_2$ are necessarily related by a linear fractional transformation
\begin{equation*}
     g_1=\frac{\alpha g_2+\beta}{\gamma g_2+\delta},
\end{equation*}
it is easily seen that all functions have indeed the good properties.

The second idea is to use a very general statement on conformal mapping. Namely, any conformal mapping which maps the unit disc onto a Jordan domain (the domain $\mathcal G_\mathcal L$) with analytic boundary (our curve $\mathcal L$) can be extended to a univalent function in a larger disc, see \cite[Sec.~1.6]{Du-83}. As the extension is univalent, it becomes obvious that the derivative $g'$ in the denominator of $f$ cannot vanish.
\end{proof}

\subsection{Proof of Theorem \ref{thm:second_main}}
\label{subsec:proof_thm_3/4_index0}

Our main idea here is to reformulate the initial boundary condition \eqref{eq:bvp_RCshift_3/4} as \eqref{eq:other_way_BVP}, with the help of a function $f$ which is analytic in $\mathcal{G}_\mathcal L$, admits finite limits on $\mathcal L$ and satisfies on $\mathcal L$ the decoupling condition \eqref{eq:decoupling_d_f}. Using Lemma \ref{lem:BVP} and Theorem \ref{thm:solving_decoupling_condition}, we deduce that $f(y)D(y)$ is analytic in $\mathcal{G}_{\mathcal L}$ and has finite limits on $\mathcal L$. As a consequence, $f(y)D(y)$ satisfies a Riemann-Carleman BVP with index zero (in the sense of Definition \ref{def:index}). Similarly to Section \ref{sub:proof-thmbvpM1} and using again a conformal gluing function, we transform the latter BVP into a Riemann-Hilbert BVP on an open contour, whose solution is 
\begin{equation}
\label{eq:resolution_BVP0}
     D(y)f(y)= \frac{1}{2i\pi}\int_{\mathcal{L}}\frac{zf(z)}{\sqrt{\widetilde{d}(z)}}\frac{w'(z)}{w(z)-w(y)}\dz+c, 
\end{equation}
where $c$ is constant in $y$, but may depend on $t$ (as recalled in Theorem \ref{thm:Sol-BVP} from Appendix \ref{app:RHBVP}, the solutions to a BVP of index zero are determined up to one constant). Notice that $f$ cancels at $y_2$ (the unique pole of $w$) and the integral in the right-hand side of \eqref{eq:resolution_BVP0} as well, it follows that $c=0$.

We now simplify the integrand in \eqref{eq:resolution_BVP0}. First, noting that $h$ satisfies the simple differential equation $h'=\frac{-h}{\sqrt{z^2-1}}$, we obtain with our notation \eqref{eq:def_hat_w}
\begin{equation*}
     f=\frac{g}{g'}=\frac{h(\widehat{w})}{\widehat{w}'h'(\widehat{w})}=-\frac{\sqrt{\widehat{w}^2-1}}{\widehat{w}'}=-\frac{\sqrt{(w-w(Y(x_1)))(w-w(Y(x_2)))}}{w'}.
\end{equation*}
Furthermore, the conformal gluing function $w$ satisfies the following differential equation
\begin{equation}
\label{eq:differential-eq-w}
     \widetilde{d}(z)w'(z)^2=(w(z)-w(Y(x_1)))(w(z)-w(Y(x_2)))(w(z)-w(y_1)),
\end{equation}
see \cite[Sec.~5.5.2.2]{FaIaMa-17}. Taking the square root of \eqref{eq:differential-eq-w} in the neighborhood of $[y_2,y_3]\cap \mathcal G_\mathcal L$ gives
\begin{equation*}
     -\sqrt{\widetilde{d}(z)}w'(z)=\sqrt{(w(z)-w(Y(x_1)))(w(z)-w(Y(x_2)))(w(z)-w(y_1))},
\end{equation*}
as $w$ is decreasing on $[y_2,y_3]\cap \mathcal G_\mathcal L$. It follows that 
\begin{equation*}
     \frac{f(z)}{\sqrt{\widetilde{d}(z)}}=\frac{1}{\sqrt{w(z)-w\left(y_1\right)}}.
\end{equation*}
The proof of Theorem \ref{thm:second_main} is complete.

\begin{rem}
\label{rem:pole_y2}
The differential equation \eqref{eq:differential-eq-w} is only true for the conformal gluing function $w$ whose expression is given in \eqref{eq:expression_gluing}, with a pole at $y_2$. If instead we have at hand a function $w$ with a pole at $y_0\neq y_2$ (for example $y_0=0$, as in Lemma \ref{lem:list_conformal_gluing_functions}), we can consider $w_{y_{2}}=\frac{1}{w-w(y_2)}$, which instead of \eqref{eq:differential-eq-w} satisfies the differential equation
\begin{equation*}
     \widetilde{d}(z)w_{y_{2}}'(z)^2=\widetilde{d}'(y_2)w'(y_2)(w_{y_{2}}(z)-w_{y_{2}}(Y_0(x_1))) (w_{y_{2}}(z)-w_{y_{2}}(Y_0(x_2))) (w_{y_{2}}(z)-w_{y_{2}}(y_1)).
\end{equation*}
\end{rem}

\footnotesize
\bibliographystyle{plain}
\bibliography{bibl}

\normalsize

\appendix

\section{Expression and properties of conformal gluing functions}
\label{app:expression_Gluing}

A crucial ingredient in our main results (Theorems \ref{thm:first_main} and \ref{thm:second_main}) is the function $w(y)$, which we interpret as a conformal mapping from the domain $\mathcal G_\mathcal L$ onto a complex plane cut along an interval, see Section~\ref{subsec:Roots_of_the_kernel}. In this appendix, we recall from \cite{Ra-12,BeBMRa-17} an explicit expression as well as some analytic properties of this function, first in the finite group case, then for infinite group models.

Let us recall that if $w$ is a suitable mapping, then any $\frac{\alpha w+\beta}{\gamma w+\delta}$ is also a suitable mapping, as soon as $\alpha\delta-\beta\gamma\neq 0$. Therefore, all expressions hereafter are given up to such a fractional linear transform.

\subsection{Finite group models}

We start by giving an expression of the conformal mapping $w(y)$ for the Kreweras trilogy of Figure~\ref{fig:symmetric_models_finite}. Let $W=W(t)$ (resp.\ $Z=Z(t)$) be the unique power series (resp.\ the unique power series with no constant term) satisfying 
\begin{equation}
\label{eq:W_Z}
     W=t(2+W^3)\qquad \text{and}\qquad Z=t\frac{1-2Z+6Z^2-2Z^3+Z^4}{(1-Z)^2}.
\end{equation}

\begin{lem}
\label{lem:list_conformal_gluing_functions}
Let $W$ and $Z$ as in \eqref{eq:W_Z}. The function
\begin{equation*}
     w(y) = \left(\frac{1}{y}-\frac{1}{W}\right)\sqrt{1-yW^2}
\end{equation*}
is a conformal mapping for Kreweras model. Likewise, a conformal mapping for reverse Kreweras model is given by
\begin{equation*}
     w(y) = \frac{-ty^3+y^2+t}{2yt} -\frac{2y^2-yW^2-W}{2yW}\sqrt{1-yW(W^3+4)/4+y^2W^2/4}.
\end{equation*}
Finally, a conformal mapping for double Kreweras model is 
\begin{multline*}
     w(y) = \sqrt{1-2yZ(1+Z^2)/(1-Z)^2+Z^2y^2}\frac{(Z(1-Z)+2yZ-(1-Z)y^2)}{2yZ(1-Z)(1+y)}\\
     +\frac{Z(1-Z)^2-Z^2(-1+2Z+Z^2)y+(1-2Z+7Z^2-4Z^3)y^2-Z(1-Z)^2y^3}{2y(1+y)Z(1-Z)^2}.
\end{multline*}
\end{lem}

Notice that the functions $w$ given in Lemma \ref{lem:list_conformal_gluing_functions} all have a pole at $y=0$.

\begin{proof}
Expressions for $w$ are given in \cite[Thm.~3 (iii)]{Ra-12}, but some quantities in the latter statement (namely $\alpha$, $\beta$, $\delta$ and $\gamma$, all depending on $t$) are not totally explicit. So to derive the above expressions of $w$, we will rather use a combination of the works \cite{BMMi-10} and \cite{BeBMRa-17}. Indeed, algebraic expressions of $Q(0,y)$ in terms of $y$ and $t$ are obtained in \cite{BMMi-10} for the three Kreweras models (see Prop.~13, Prop.~14 and Prop.~15 there). On the other hand, an alternative formulation of $Q(0,y)$ as a rational function of $w(y)$, $y$ and $t$ is derived in \cite{BeBMRa-17} (see Thm.~23 and Table 8 there). The formulas of Lemma \ref{lem:list_conformal_gluing_functions} are obtained by equating the two expressions.
\end{proof}

An expression for $w(y)$ for Gessel's model is obtained in \cite[Thm.~7]{KuRa-11}.

\subsection{Infinite group models}

In the infinite group case, the function $w$ is not algebraic anymore (it is even non-D-finite, see \cite[Thm.~2]{Ra-12}). As $\mathcal{L}$ is a quartic curve \cite[Thm.~5.3.3 (i)]{FaIaMa-17}, $w$ can be expressed in terms of Weierstrass' elliptic functions (see \cite[Sec.~5.5.2.1]{FaIaMa-17} or \cite[Thm.~6]{Ra-12}): 
\begin{lem}[\cite{FaIaMa-17,Ra-12,BeBMRa-17}]
\label{lem:conformal_gluing_functions_infinite_group}
The function $w$ defined by
\begin{equation}
\label{eq:expression_gluing}
     w(y)=\wp_{1,3}\Big(-\frac{\omega_1+\omega_2}{2}+\wp_{1,2}^{-1}(f(y))\Big)
\end{equation}
is a conformal mapping for the domain $\mathcal G_\mathcal L$, and has in this domain a unique (and simple) pole, located at $y_2$. The function $w$ admits a meromorphic continuation on $\mathbb C\setminus [y_3,y_4]$. It is D-algebraic in $y$ and in $t$.
\end{lem}
The differential algebraicity is shown in \cite[Thm.~33]{BeBMRa-17}. The remaining properties stated in Lemma~\ref{lem:conformal_gluing_functions_infinite_group} come from \cite{FaIaMa-17,Ra-12}, see e.g.\ \cite[Thm.~6 and Rem.~7]{Ra-12}.

Let us now comment on the expression \eqref{eq:expression_gluing}, following the discussion in \cite[Sec.~5.2]{BeBMRa-17}. First, $f(y)$ is a rational function of $y$ whose coefficients are algebraic functions of $t$:
\begin{equation*}
     f(y) = \left\{\begin{array}{ll}
     \displaystyle \frac{\widetilde d''(y_4)}{6}+\frac{\widetilde d'(y_4)}{y-y_4} & \text{if } y_4\neq \infty,\medskip\\
     \displaystyle\frac{\widetilde d''(0)}{6}+\frac{\widetilde d'''(0)y}{6}& \text{if } y_4=\infty,
     \end{array}\right.
\end{equation*} 
where $\widetilde d(y)$ is the discriminant \eqref{eq:discriminants} and $y_4$ is one of its roots.

The next ingredient in \eqref{eq:expression_gluing} is Weierstrass' elliptic function $\wp$, with periods $\omega_1$ and $\omega_2$:
\begin{equation*}
     \wp(z)= \wp(z, \omega_1, \omega_2) = \frac 1 {z^2} +\sum_{(i,j) \in \mathbb Z^2
  \setminus\{(0,0)\}} \left( \frac 1
  {(z-i\omega_1-j\omega_2)^2}-\frac1{(i\omega_1+j\omega_2)^2}\right).
\end{equation*}
Then $\wp_{1,2}(z)$ (resp.\ $\wp_{1,3}(z)$) is the Weierstrass function
with periods $\omega_1$ and $\omega_2$ (resp.\ $\omega_1$ and $\omega_3$) defined by:
\begin{equation*}
   \omega_1 = i\int_{y_1}^{y_2} \frac{\text{d} y}{\sqrt{- \widetilde d(y)}},\qquad
     \omega_2 = \int_{y_2}^{y_3} \frac{\text{d} y}{\sqrt{ \widetilde d(y)}},\qquad
     \omega_3 = \int_{Y(x_1)}^{y_1} \frac{\text{d} y}{\sqrt{ \widetilde d(y)}}.
\end{equation*}
These definitions make sense thanks to the properties
  of the $y_i$'s and $Y(x_i)$'s  (see
\cite[Sec.~5.1]{BeBMRa-17}). 
If $Y(x_1)$ is infinite (which happens if
  and only if neither $(-1,0)$ nor $(-1,1)$ are in $\mathcal S$), the
  integral defining $\omega_3$ starts at $-\infty$.
Note that $\omega_1\in i\mathbb R_+$
and $\omega_2,\omega_3\in \mathbb R_+$. 

Finally, as the Weierstrass function is not
injective on $\mathbb C$, we need to clarify our definition of
$\wp_{1,2}^{-1}$ in \eqref{eq:expression_gluing}. The function
$\wp_{1,2}$ is two-to-one on the fundamental parallelogram
$[0,\omega_1)+[0,\omega_2)$ (because $\wp(z)=\wp(-z+\omega_1+\omega_2)$),
but is one-to-one when
restricted to a half-parallelogram---more precisely,  when restricted to the open
  rectangle $(0,\omega_1)+(0,
  \omega_2/2)$ together with the three boundary segments  $[0, \omega_1/2]$,
   $[0, \omega_2/2]$ and $\omega_2/2+[0, \omega_1/2]$. 
  We choose the
determination of $\wp_{1,2}^{-1}$ in this set.

\section{Riemann-Hilbert BVP}
\label{app:RHBVP}

In the way of proving our main results (Theorems \ref{thm:first_main} and \ref{thm:second_main}), a crucial ingredient is the BVP with shift of Lemma \ref{lem:BVP}. It is solved by reduction to a more classical Riemann BVP (Sections \ref{sub:proof-thmbvpM1} and \ref{subsec:proof_thm_3/4_index0}). In this appendix we present the main formulas used to solve the latter, so as to render our paper self-contained. Our main references are the books of Gakhov \cite[Chap.~2]{Ga-90} and Lu \cite[Chap.~4]{Lu-93}.

\begin{figure}[htb]
\centering
\begin{tikzpicture}
\begin{scope}[scale=0.3]
\draw [dblue!90, thick] plot [smooth] coordinates {(-5,-3) (-3,1) (1,1) (4,4)};
\draw [->, dblue!90, thick] (-3,1)--(-2.9,1.1);
\draw[thick, dblue!90, fill=dblue!90](-5,-3) circle (0.15 cm)  node[below]{\tiny{$a$}};
\draw[thick, dblue!90, fill=dblue!90](4,4) circle (0.15 cm)  node[below]{\tiny{$b$}};
\draw [dblue!90] (-4.75,0.5) node{\small{$\mathcal{U}$}};
\draw [->, dgreen!90, thick] (-0.45,2.45)--(0.9,1.1);
\draw [dgreen!90, thick] (1,2.5) node{\small{$\Phi^+$}};
\draw [->, dgreen!90, thick] (2.45,-0.45)--(1.1,0.9);
\draw [dgreen!90, thick] (1,-0.5) node{\small{$\Phi^-$}};

\end{scope}
\end{tikzpicture}
\caption{Left and right limits on the open contour $\mathcal U$}
\label{fig:BVP-open}
\end{figure}
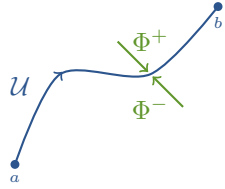

Suppose that $\mathcal{U}$ is an open, smooth, non-intersecting, oriented curve from $a$ to $b$, see Figure~\ref{fig:BVP-open} for an example. Throughout, for $z\in\mathcal U$, we will denote by $\Phi^+(z)$ (resp.\ $\Phi^-(z)$) the limit of a function $\Phi$ as $y\to z$ from the left (resp.\ right) of $\mathcal{U}$, see again Figure \ref{fig:BVP-open}.

\begin{defn}[Riemann BVP]
\label{def:BVP}
Let $\mathcal{U}$ be as above. A function $\Phi$ satisfies a BVP on $\mathcal{U}$ if:
\begin{itemize}
     \item $\Phi$ is sectionally analytic, i.e., analytic in $\mathbb{C}\setminus\mathcal{U}$;
     \item $\Phi$ has finite degree at $\infty$ (the only singularity at $\infty$ is a pole of finite order), and $\Phi$ is bounded in the vicinity of the extremities $a$ and $b$;
     \item $\Phi$ has left limits $\Phi^+$ and right limits $\Phi^-$ on $\mathcal U$;
     \item $\Phi$ satisfies the following boundary condition
\begin{equation}
\label{eq:RBVP}
     \Phi^+(z)=G(z)\Phi^-(z)+g(z), \quad z\in\mathcal{U},
\end{equation}
where $G$ and $g$ are H\"{o}lder functions on $\mathcal{U}$, and $G$ does not vanish on $\mathcal{U}$. 
\end{itemize}
\end{defn}

Let us recall the so-called Sokhotski-Plemelj formulas, which represent a crucial tool to solve the BVP of Definition \ref{def:BVP}.  

\begin{proposition}[Sokhotski-Plemelj formulas]
\label{proposition:Sokhotski-Plemelj}
Let $\mathcal{U}$ be as above, and let $f$ be a H\"{o}lder function on $\mathcal{U}$. The contour integral
\begin{equation*}
     F(z)=\frac{1}{2i\pi} \int_{\mathcal{U}} \frac{f(u)}{u-z}\du
\end{equation*}
is sectionally analytic on $\mathbb C\setminus \mathcal U$. Its left and right limit values $F^+$ and $F^-$ are H\"{o}lder functions on $\mathcal{U}$ and satisfy, for $z\in\mathcal{U}$,
\begin{equation*}
     F^\pm(z) =\pm \frac{1}{2}f(z) +\frac{1}{2i\pi}\int_\mathcal{U}\frac{f(u)}{u-z}\du,
\end{equation*}
where the very last integral is understood in the sense of Cauchy-principal value, see \cite[Chap.~1, Sec.~12]{Ga-90}.
This is equivalent to the following equations on $\mathcal U$:
\begin{equation}
\left\{
\begin{array}{l r l l|}
     F^+(z) -F^-(z)&\hspace{-2mm}=&\hspace{-2mm}f(z),\smallskip\\
     F^-(z) +F^-(z)&\hspace{-2mm}=&\hspace{-2mm}\displaystyle\frac{1}{i\pi}\int_\mathcal{L}\frac{f(u)}{u-z}\du.
\end{array}
\right.
\end{equation}
\end{proposition}

We also define the following important quantity:
\begin{defn}[Index] 
\label{def:index}
Let $\mathcal{U}$ be as above and let $G$ be the function (continuous on $\mathcal{U}$) as in \eqref{eq:RBVP}. The index $\chi$ of the BVP of Definition \ref{def:BVP} is
\begin{equation*}
     \chi=\Ind_{\mathcal U}G=\frac{1}{2\pi}[\arg G]_\mathcal{U}=\frac{1}{2i\pi}[\log G]_\mathcal{U}=\frac{1}{2i\pi}\int_{\mathcal U} \frac{G'(u)}{G(u)}\du.
\end{equation*}
\end{defn}
Plainly, $\chi$ represents the variation of argument of $G(u)$, when $u$ moves along the contour $\mathcal{U}$ in the positive direction.

The main result is the following, see \cite[Chap.~4, Thm.~2.1.2]{Lu-93}:

\begin{thm}[Solution of Riemann-Hilbert BVP]
\label{thm:Sol-BVP}
Let $\mathcal U$ be as above. The solution of the BVP of Definition \ref{def:BVP} is given by, for $z\notin \mathcal{U}$,
\begin{equation}
\Phi(z)=
\left\{
\begin{array}{l l}
X(z)\psi(z)+X(z)P_\chi(z) & \text{ if } \chi\geq0,  \\
X(z)\psi(z) & \text{ if } \chi=-1, \\
X(z)\psi(z) & \text{ if } \chi <-1  \text{ and if the solvability conditions below hold:}
\end{array}
\right.
\end{equation}
\begin{equation*}
     \frac{1}{2i\pi}\int_\mathcal{U}\frac{g(u)u^{k-1}}{X^+(u)}\du=0, \quad k=1,\ldots,-\chi-1,
\end{equation*}
where $P_\chi$ is an arbitrary polynomial of degree $\chi$, and
\begin{equation*}
\left\{
\begin{array}{r c l}
X(z)&=&\displaystyle (z-b)^{-\chi}\exp \Gamma(z),\medskip\\
X^+(z)&=&\displaystyle(z-b)^{-\chi}\exp \Gamma^+(z),\medskip\\
\Gamma(z)&=&\displaystyle\frac{1}{2i\pi}\int_\mathcal{U}\frac{\log G(u)}{u-z}\du,\medskip\\
\psi(z)&=&\displaystyle\frac{1}{2i\pi}\int_\mathcal{U}\frac{g(u)}{X^+(u)(u-z)}\du.
\end{array}
\right.
\end{equation*}

\end{thm}

\section{Proof of Lemma \ref{lem:functional_equation_sym}}
\label{app:proof_lem1}

The decomposition in \eqref{eq:equation_cut3parts} expresses $C(x,y)$ as a sum of three generating functions. Thanks to the symmetry of the step set and the fact that the starting point lies on the diagonal, $\widehat{U}(x,y)=\widehat{L}(y,x)$ and $C(x,y)$ is written as the sum $\widehat{L}(x,y)+\widehat{D}(x,y)$ of two unknowns.
We further introduce the generating functions
\begin{equation*}
     \widehat{D}^{\ell}(x,y)=\sum_{n\geq0,i\geq0}c_{i,i-1}(n)x^{i}y^{i-1}t^n\quad\text{and}     
     \quad
     \widehat{D}^{u}(x,y)=\sum_{n\geq0,i\geq0}c_{i-1,i}(n)x^{i-1}y^it^n,
\end{equation*}
which respectively count walks ending on the lower (resp.\ upper) diagonal, see Figure \ref{fig:some_sections}. In this section, we consider walks starting on the diagonal and ending anywhere in the three-quadrant $\mathcal C$.

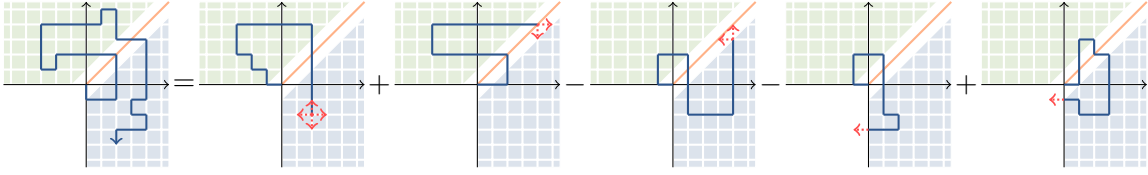
\begin{figure}[!h]
\begin{center}
\begin{tikzpicture}
\begin{scope}[scale=0.2]
\draw[white, fill=dblue!15] (0,-1) -- (5.5,4.5) -- (5.5,-5.5) -- (0,-5.5);
\draw[white, fill=dgreen!15] (-1,0) -- (-5.5,0) -- (-5.5,5.5) -- (4.5,5.5);
\draw[white, thick] (0,-5.5) grid (5.5,5.5);
\draw[white, thick] (-5.5,0) grid (0,5.5);
\draw[Apricot!90, thick] (0,0) -- (5.5,5.5);
\draw[->] (0,-5.5) -- (0,5.5);
\draw[->] (-5.5,0) -- (5.5,0);
\draw[thick, dblue!90] (0,0) -- (0,-1);
\draw[thick, dblue!90] (0,-1) -- (2,-1);
\draw[thick, dblue!90] (2,-1) -- (2,2);
\draw[thick, dblue!90] (2,2) -- (-2,2);
\draw[thick, dblue!90] (-2,2) -- (-2,1);
\draw[thick, dblue!90] (-2,1) -- (-3,1);
\draw[thick, dblue!90] (-3,1) -- (-3,4);
\draw[thick, dblue!90] (-3,4) -- (1,4);
\draw[thick, dblue!90] (1,4) -- (1,5);
\draw[thick, dblue!90] (1,5) -- (2,5);
\draw[thick, dblue!90] (2,5) -- (2,3);
\draw[thick, dblue!90] (2,3) -- (4,3);
\draw[thick, dblue!90] (4,3) -- (4,-1);
\draw[thick, dblue!90] (4,-1) -- (3,-1);
\draw[thick, dblue!90] (3,-1) -- (3,-2);
\draw[thick, dblue!90] (3,-2) -- (4,-2);
\draw[thick, dblue!90] (4,-2) -- (4,-3);
\draw[thick, dblue!90] (4,-3) -- (2,-3);
\draw[thick, dblue!90, ->] (2,-3) -- (2,-4);
\draw (6.5,0) node{$=$};
\end{scope}

\begin{scope}[scale=0.2, xshift=13cm]
\draw[white, fill=dblue!15] (0,-1) -- (5.5,4.5) -- (5.5,-5.5) -- (0,-5.5);
\draw[white, fill=dgreen!15] (-1,0) -- (-5.5,0) -- (-5.5,5.5) -- (4.5,5.5);
\draw[white, thick] (0,-5.5) grid (5.5,5.5);
\draw[white, thick] (-5.5,0) grid (0,5.5);
\draw[Apricot!90, thick] (0,0) -- (5.5,5.5);
\draw[->] (0,-5.5) -- (0,5.5);
\draw[->] (-5.5,0) -- (5.5,0);
\draw[thick, dblue!90] (0,0) -- (-1,0);
\draw[thick, dblue!90] (-1,0) -- (-1,1);
\draw[thick, dblue!90] (-1,1) -- (-2,1);
\draw[thick, dblue!90] (-2,1) -- (-2,2);
\draw[thick, dblue!90] (-2,2) -- (-3,2);
\draw[thick, dblue!90] (-3,2) -- (-3,4);
\draw[thick, dblue!90] (-3,4) -- (2,4);
\draw[thick, dblue!90] (2,4) -- (2,-2);
\draw[thick, red!70, ->, densely dotted] (2,-2) -- (1,-2);
\draw[thick, red!70, ->, densely dotted] (2,-2) -- (3,-2);
\draw[thick, red!70, ->, densely dotted] (2,-2) -- (2,-3);
\draw[thick, red!70, ->, densely dotted] (2,-2) -- (2,-1);
\draw (6.5,0) node{$+$};
\end{scope}

\begin{scope}[scale=0.2, xshift=26cm]
\draw[white, fill=dblue!15] (0,-1) -- (5.5,4.5) -- (5.5,-5.5) -- (0,-5.5);
\draw[white, fill=dgreen!15] (-1,0) -- (-5.5,0) -- (-5.5,5.5) -- (4.5,5.5);
\draw[white, thick] (0,-5.5) grid (5.5,5.5);
\draw[white, thick] (-5.5,0) grid (0,5.5);
\draw[Apricot!90, thick] (0,0) -- (5.5,5.5);
\draw[->] (0,-5.5) -- (0,5.5);
\draw[->] (-5.5,0) -- (5.5,0);
\draw[thick, dblue!90] (0,0) -- (2,0);
\draw[thick, dblue!90] (2,0) -- (2,2);
\draw[thick, dblue!90] (2,2) -- (-3,2);
\draw[thick, dblue!90] (-3,2) -- (-3,4);
\draw[thick, dblue!90] (-3,4) -- (4,4);
\draw[thick, red!70, ->, densely dotted] (4,4) -- (5,4);
\draw[thick, red!70, ->, densely dotted] (4,4) -- (4,3);
\draw (6.5,0) node{$-$};
\end{scope}

\begin{scope}[scale=0.2, xshift=39cm]
\draw[white, fill=dblue!15] (0,-1) -- (5.5,4.5) -- (5.5,-5.5) -- (0,-5.5);
\draw[white, fill=dgreen!15] (-1,0) -- (-5.5,0) -- (-5.5,5.5) -- (4.5,5.5);
\draw[white, thick] (0,-5.5) grid (5.5,5.5);
\draw[white, thick] (-5.5,0) grid (0,5.5);
\draw[Apricot!90, thick] (0,0) -- (5.5,5.5);
\draw[->] (0,-5.5) -- (0,5.5);
\draw[->] (-5.5,0) -- (5.5,0);
\draw[thick, dblue!90] (0,0) -- (-1,0);
\draw[thick, dblue!90] (-1,0) -- (-1,2);
\draw[thick, dblue!90] (-1,2) -- (1,2);
\draw[thick, dblue!90] (1,2) -- (1,-2);
\draw[thick, dblue!90] (1,-2) -- (4,-2);
\draw[thick, dblue!90] (4,-2) -- (4,3);
\draw[thick, red!70, ->, densely dotted] (4,3) -- (4,4);
\draw[thick, red!70, ->, densely dotted] (4,3) -- (3,3);
\draw (6.5,0) node{$-$};
\end{scope}

\begin{scope}[scale=0.2, xshift=52cm]
\draw[white, fill=dblue!15] (0,-1) -- (5.5,4.5) -- (5.5,-5.5) -- (0,-5.5);
\draw[white, fill=dgreen!15] (-1,0) -- (-5.5,0) -- (-5.5,5.5) -- (4.5,5.5);
\draw[white, thick] (0,-5.5) grid (5.5,5.5);
\draw[white, thick] (-5.5,0) grid (0,5.5);
\draw[Apricot!90, thick] (0,0) -- (5.5,5.5);
\draw[->] (0,-5.5) -- (0,5.5);
\draw[->] (-5.5,0) -- (5.5,0);
\draw[thick, dblue!90] (0,0) -- (-1,0);
\draw[thick, dblue!90] (-1,0) -- (-1,2);
\draw[thick, dblue!90] (-1,2) -- (1,2);
\draw[thick, dblue!90] (1,2) -- (1,-2);
\draw[thick, dblue!90] (1,-2) -- (2,-2);
\draw[thick, dblue!90] (2,-2) -- (2,-3);
\draw[thick, dblue!90] (2,-3) -- (0,-3);
\draw[thick, red!70, ->, densely dotted] (0,-3) -- (-1,-3);
\draw (6.5,0) node{$+$};
\end{scope}

\begin{scope}[scale=0.2, xshift=65cm]
\draw[white, fill=dblue!15] (0,-1) -- (5.5,4.5) -- (5.5,-5.5) -- (0,-5.5);
\draw[white, fill=dgreen!15] (-1,0) -- (-5.5,0) -- (-5.5,5.5) -- (4.5,5.5);
\draw[white, thick] (0,-5.5) grid (5.5,5.5);
\draw[white, thick] (-5.5,0) grid (0,5.5);
\draw[Apricot!90, thick] (0,0) -- (5.5,5.5);
\draw[->] (0,-5.5) -- (0,5.5);
\draw[->] (-5.5,0) -- (5.5,0);
\draw[thick, dblue!90] (0,0) -- (1,0);
\draw[thick, dblue!90] (1,0) -- (1,3);
\draw[thick, dblue!90] (1,3) -- (2,3);
\draw[thick, dblue!90] (2,3) -- (2,2);
\draw[thick, dblue!90] (2,2) -- (3,2);
\draw[thick, dblue!90] (3,2) -- (3,-2);
\draw[thick, dblue!90] (3,-2) -- (1,-2);
\draw[thick, dblue!90] (1,-2) -- (1,-1);
\draw[thick, dblue!90] (1,-1) -- (0,-1);
\draw[thick, red!70, ->, densely dotted] (0,-1) -- (-1,-1);
\end{scope}

\end{tikzpicture}
\end{center}
\caption{Different ways to end in the lower part (example of the simple walk)}\label{fig:diagtolowerSW}
\end{figure}

Thereafter, $c_{i,j}(n)$ is counting walks from $(i_0,i_0)$ to $(i,j)$ in $n$ steps. Classically \cite{BMMi-10}, we construct a walk by adding a new step at the end of the walk at each stage. We first derive a functional equation for $\widehat{L}(x,y)$ by taking into account all possibilities of ending in the lower part: 
\begin{itemize}
     \item we may add a step from $\widehat{\mathcal{S}}$ (recall that $\widehat{\mathcal{S}}$ is the step set before the change of variable $\varphi$) to walks ending in the lower part, yielding in \eqref{eq:fct-eq-L(x,y)} the term $t(\sum_{(i,j)\in\widehat{\mathcal{S}}} x^i y^j )\widehat{L}(x,y)$, see the second picture on Figure \ref{fig:diagtolowerSW} in the particular case of the simple walk; 
     \item walks coming from the diagonal also need to be counted up, giving rise in \eqref{eq:fct-eq-L(x,y)} to the term $t(\delta_{1,0}x+\delta_{0,-1}y^{-1})\widehat{D}(x,y)$ (third picture on Figure \ref{fig:diagtolowerSW}); 
     \item on the other hand, walks going out of the three-quarter plane need to be removed, yielding the terms $t( \delta_{-1,0}x^{-1}+\delta_{0,1}y) \widehat{D}^\ell(x,y)$ (the lower diagonal) and $t(\delta_{-1,0}x^{-1}+\delta_{-1,-1}x^{-1}y^{-1}) \widehat{L}_{0-}(y^{-1})$ (negative $y$-axis), see the fourth and fifth pictures on Figure \ref{fig:diagtolowerSW};
     \item we finally add the term $t\delta_{-1,0}x^{-1}\sum_{n\geq 0}c_{0,-1}(n)y^{-1}t^n$ which was subtracted twice, corresponding to the rightmost picture on Figure \ref{fig:diagtolowerSW}. 
\end{itemize}     
We end up with a first functional equation:
\begin{multline}
\label{eq:fct-eq-L(x,y)}
\widehat{L}(x,y)= t\sum_{(i,j)\in\widehat{\mathcal{S}}} x^i y^j \widehat{L}(x,y)
+ t(\delta_{1,0}x+\delta_{0,-1}y^{-1})\widehat{D}(x,y) 
-t( \delta_{-1,0}x^{-1}+\delta_{0,1}y) \widehat{D}^\ell(x,y)
\\
-t( \delta_{-1,0}x^{-1}+\delta_{-1,-1}x^{-1}y^{-1}) \widehat{L}_{0-}(y^{-1})
+t( \delta_{-1,0}x^{-1})\sum_{n\geq 0}c_{0,-1}(n)y^{-1}t^n.
\end{multline}
We now prove the second equation
\begin{multline}
\label{eq:fct-eq-D(x,y)}
\widehat{D}(x,y)=x^{i_0}y^{i_0}
+t( \delta_{1,1}xy+\delta_{-1,-1}x^{-1}y^{-1})\widehat{D}(x,y)-t\delta_{-1,-1}x^{-1}y^{-1}\widehat{D}(0,0)\\
+2t( \delta_{-1,0}x^{-1}+\delta_{0,1}y) \widehat{D}^\ell(x,y)
-2t\delta_{-1,0}x^{-1}\sum_{n\geq 0}c_{0,-1}(n)y^{-1}t^n,
\end{multline}
and remark that by plugging in \eqref{eq:fct-eq-D(x,y)} into \eqref{eq:fct-eq-L(x,y)} we get \eqref{eq:functional_equation_sym}, thereby completing the proof of Lemma~\ref{lem:functional_equation_sym}.
 
This second equation \eqref{eq:fct-eq-D(x,y)} is obtained by writing all possibilities of ending on the diagonal, as illustrated on Figure \ref{fig:diagtodiagSW} for simple walks: 
\begin{itemize}
     \item we first count the empty walk, giving the term $x^{i_0}y^{i_0}$; 
     \item we add the walks remaining on the diagonal $t( \delta_{1,1}xy+\delta_{-1,-1}x^{-1}y^{-1})\widehat{D}(x,y)$, the walks ending on the diagonal coming from the upper part $t(\delta_{0,-1}y^{-1}+\delta_{1,0}x)\widehat{D}^u(x,y)$ and those coming from the lower part $t(\delta_{-1,0}x^{-1}+\delta_{0,1}y) \widehat{D}^\ell(x,y)$;
     \item finally, walks going out of the domain need to be removed, giving $t\delta_{-1,-1}x^{-1}y^{-1}\widehat{D}(0,0)$, $t\delta_{0,-1}y^{-1}\sum_{n\geq 0}c_{-1,0}(n)x^{-1}t^n$ and $t\delta_{-1,0}x^{-1}\sum_{n\geq 0}c_{0,-1}(n)y^{-1}t^n$. 
\end{itemize}
Thanks to the symmetry of the step set, the number of walks coming from the upper part is the same as the number of walks coming from the lower part.

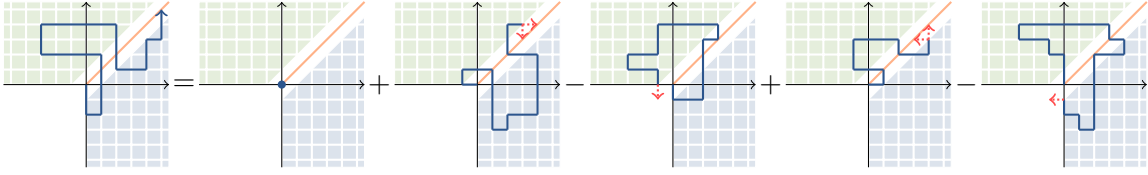
\begin{figure}[!h]
\begin{center}
\begin{tikzpicture}
\begin{scope}[scale=0.2]
\draw[white, fill=dblue!15] (0,-1) -- (5.5,4.5) -- (5.5,-5.5) -- (0,-5.5);
\draw[white, fill=dgreen!15] (-1,0) -- (-5.5,0) -- (-5.5,5.5) -- (4.5,5.5);
\draw[white, thick] (0,-5.5) grid (5.5,5.5);
\draw[white, thick] (-5.5,0) grid (0,5.5);
\draw[Apricot!90, thick] (0,0) -- (5.5,5.5);
\draw[->] (0,-5.5) -- (0,5.5);
\draw[->] (-5.5,0) -- (5.5,0);
\draw[thick, dblue!90] (0,0) -- (0,-2);
\draw[thick, dblue!90] (0,-2) -- (1,-2);
\draw[thick, dblue!90] (1,-2) -- (1,2);
\draw[thick, dblue!90] (1,2) -- (-3,2);
\draw[thick, dblue!90] (-3,2) -- (-3,4);
\draw[thick, dblue!90] (-3,4) -- (2,4);
\draw[thick, dblue!90] (2,4) -- (2,1);
\draw[thick, dblue!90] (2,1) -- (4,1);
\draw[thick, dblue!90] (4,1) -- (4,3);
\draw[thick, dblue!90] (4,3) -- (5,3);
\draw[thick, dblue!90, ->] (5,3) -- (5,5);
\draw (6.5,0) node{$=$};
\end{scope}

\begin{scope}[scale=0.2, xshift=13cm]
\draw[white, fill=dblue!15] (0,-1) -- (5.5,4.5) -- (5.5,-5.5) -- (0,-5.5);
\draw[white, fill=dgreen!15] (-1,0) -- (-5.5,0) -- (-5.5,5.5) -- (4.5,5.5);
\draw[white, thick] (0,-5.5) grid (5.5,5.5);
\draw[white, thick] (-5.5,0) grid (0,5.5);
\draw[Apricot!90, thick] (0,0) -- (5.5,5.5);
\draw[->] (0,-5.5) -- (0,5.5);
\draw[->] (-5.5,0) -- (5.5,0);
\draw[thick, dblue!90, fill=dblue!90] (0,0) circle (0.2 cm);
\draw (6.5,0) node{$+$};
\end{scope}

\begin{scope}[scale=0.2, xshift=26cm]
\draw[white, fill=dblue!15] (0,-1) -- (5.5,4.5) -- (5.5,-5.5) -- (0,-5.5);
\draw[white, fill=dgreen!15] (-1,0) -- (-5.5,0) -- (-5.5,5.5) -- (4.5,5.5);
\draw[white, thick] (0,-5.5) grid (5.5,5.5);
\draw[white, thick] (-5.5,0) grid (0,5.5);
\draw[Apricot!90, thick] (0,0) -- (5.5,5.5);
\draw[->] (0,-5.5) -- (0,5.5);
\draw[->] (-5.5,0) -- (5.5,0);
\draw[thick, dblue!90] (0,0) -- (-1,0);
\draw[thick, dblue!90] (-1,0) -- (-1,1);
\draw[thick, dblue!90] (-1,1) -- (1,1);
\draw[thick, dblue!90] (1,1) -- (1,-3);
\draw[thick, dblue!90] (1,-3) -- (2,-3);
\draw[thick, dblue!90] (2,-3) -- (2,-2);
\draw[thick, dblue!90] (2,-2) -- (4,-2);
\draw[thick, dblue!90] (4,-2) -- (4,2);
\draw[thick, dblue!90] (4,2) -- (2,2);
\draw[thick, dblue!90] (2,2) -- (2,4);
\draw[thick, dblue!90] (2,4) -- (3,4);
\draw[thick, red!70, ->, densely dotted] (3,4) -- (4,4);
\draw[thick, red!70, ->, densely dotted] (3,4) -- (3,3);
\draw (6.5,0) node{$-$};
\end{scope}

\begin{scope}[scale=0.2, xshift=39cm]
\draw[white, fill=dblue!15] (0,-1) -- (5.5,4.5) -- (5.5,-5.5) -- (0,-5.5);
\draw[white, fill=dgreen!15] (-1,0) -- (-5.5,0) -- (-5.5,5.5) -- (4.5,5.5);
\draw[white, thick] (0,-5.5) grid (5.5,5.5);
\draw[white, thick] (-5.5,0) grid (0,5.5);
\draw[Apricot!90, thick] (0,0) -- (5.5,5.5);
\draw[->] (0,-5.5) -- (0,5.5);
\draw[->] (-5.5,0) -- (5.5,0);
\draw[thick, dblue!90] (0,0) -- (0,-1);
\draw[thick, dblue!90] (0,-1) -- (2,-1);
\draw[thick, dblue!90] (2,-1) -- (2,3);
\draw[thick, dblue!90] (2,3) -- (3,3);
\draw[thick, dblue!90] (3,3) -- (3,4);
\draw[thick, dblue!90] (3,4) -- (-1,4);
\draw[thick, dblue!90] (-1,4) -- (-1,2);
\draw[thick, dblue!90] (-1,2) -- (-3,2);
\draw[thick, dblue!90] (-3,2) -- (-3,1);
\draw[thick, dblue!90] (-3,1) -- (-1,1);
\draw[thick, dblue!90] (-1,1) -- (-1,0);
\draw[thick, red!70, ->, densely dotted] (-1,0) -- (-1,-1);
\draw (6.5,0) node{$+$};
\end{scope}

\begin{scope}[scale=0.2, xshift=52cm]
\draw[white, fill=dblue!15] (0,-1) -- (5.5,4.5) -- (5.5,-5.5) -- (0,-5.5);
\draw[white, fill=dgreen!15] (-1,0) -- (-5.5,0) -- (-5.5,5.5) -- (4.5,5.5);
\draw[white, thick] (0,-5.5) grid (5.5,5.5);
\draw[white, thick] (-5.5,0) grid (0,5.5);
\draw[Apricot!90, thick] (0,0) -- (5.5,5.5);
\draw[->] (0,-5.5) -- (0,5.5);
\draw[->] (-5.5,0) -- (5.5,0);
\draw[thick, dblue!90] (0,0) -- (1,0);
\draw[thick, dblue!90] (1,0) -- (1,1);
\draw[thick, dblue!90] (1,1) -- (-1,1);
\draw[thick, dblue!90] (-1,1) -- (-1,3);
\draw[thick, dblue!90] (-1,3) -- (2,3);
\draw[thick, dblue!90] (2,3) -- (2,2);
\draw[thick, dblue!90] (2,2) -- (4,2);
\draw[thick, dblue!90] (4,2) -- (4,3);
\draw[thick, red!70, ->, densely dotted] (4,3) -- (3,3);
\draw[thick, red!70, ->, densely dotted] (4,3) -- (4,4);
\draw (6.5,0) node{$-$};
\end{scope}

\begin{scope}[scale=0.2, xshift=65cm]
\draw[white, fill=dblue!15] (0,-1) -- (5.5,4.5) -- (5.5,-5.5) -- (0,-5.5);
\draw[white, fill=dgreen!15] (-1,0) -- (-5.5,0) -- (-5.5,5.5) -- (4.5,5.5);
\draw[white, thick] (0,-5.5) grid (5.5,5.5);
\draw[white, thick] (-5.5,0) grid (0,5.5);
\draw[Apricot!90, thick] (0,0) -- (5.5,5.5);
\draw[->] (0,-5.5) -- (0,5.5);
\draw[->] (-5.5,0) -- (5.5,0);
\draw[thick, dblue!90] (0,0) -- (0,2);
\draw[thick, dblue!90] (0,2) -- (-1,2);
\draw[thick, dblue!90] (-1,2) -- (-1,3);
\draw[thick, dblue!90] (-1,3) -- (-3,3);
\draw[thick, dblue!90] (-3,3) -- (-3,4);
\draw[thick, dblue!90] (-3,4) -- (3,4);
\draw[thick, dblue!90] (3,4) -- (3,3);
\draw[thick, dblue!90] (3,3) -- (4,3);
\draw[thick, dblue!90] (4,3) -- (4,2);
\draw[thick, dblue!90] (4,2) -- (2,2);
\draw[thick, dblue!90] (2,2) -- (2,-3);
\draw[thick, dblue!90] (2,-3) -- (1,-3);
\draw[thick, dblue!90] (1,-3) -- (1,-2);
\draw[thick, dblue!90] (1,-2) -- (0,-2);
\draw[thick, dblue!90] (0,-2) -- (0,-1);
\draw[thick, red!70, ->, densely dotted] (0,-1) -- (-1,-1);
\end{scope}
\end{tikzpicture}
\end{center}
\caption{Different ways to end on the diagonal (example of the simple walk)}\label{fig:diagtodiagSW}
\end{figure}

\begin{rem}
A step set containing the jumps $(-1,1)$ and $(1,-1)$ would lead to two additional terms in the functional equations, namely 
\begin{equation*}
     \delta_{-1,1}x^{-1}y\sum_{n,i\geq 0}c_{i,i-2}(n)x^i y^{i-2}t^n\qquad 
     \text{and}
     \qquad \delta_{1,-1}xy^{-1} \sum_{n,j\geq 0}c_{j-2,j}(n)x^{j-2}y^jt^n,
\end{equation*} 
making the resolution much more complicated (not to say impossible, by our techniques!). Likewise, considering an asymmetric step set and/or a starting point out of the diagonal would lead to other terms in the functional equation. 
\end{rem}

\end{document}